\documentclass[10pt]{amsart}

\usepackage{amssymb}
\usepackage{graphicx}
\usepackage{hyperref}

\usepackage{amsmath,amsthm,amsfonts,float}
\usepackage[usenames, dvipsnames]{color}

\def\N{{\mathbb{N}}}
\def\Z{{\mathbb{Z}}}
\def\R{{\mathbb{R}}}
\def\Q{{\mathbb{Q}}}
\newcommand{\ssm}{\smallsetminus}

\newtheorem{example}{Example}
\newtheorem{corollary}{Corollary}
\newtheorem{prop}{Proposition}
\newtheorem{theorem}{Theorem}
\newtheorem{lemma}{Lemma}

\title{Hierarchy for groups acting on hyperbolic $\Z^n$-spaces}

\author{Andrei-Paul Grecianu}
\address{Department of Mathematical Sciences, Stevens Institute of Technology, 1 Castle Point on Hudson, Hoboken, NJ 07030, USA} 
\email{grecianu@gmail.com}

\author{Alexei Myasnikov}
\address{Department of Mathematical Sciences, Stevens Institute of Technology, 1 Castle Point on Hudson, Hoboken, NJ 07030, USA} 
\email{amiasnikov@gmail.com}

\author{Denis Serbin}
\address{Department of Mathematical Sciences, Stevens Institute of Technology, 1 Castle Point on Hudson, Hoboken, NJ 07030, USA} 
\email{d.e.serbin@gmail.com}



\date{}

\begin{document}

\begin{abstract}
In \cite{GKMS:2012}, the authors initiated a systematic study of hyperbolic $\Lambda$-metric spaces, where $\Lambda$ is an ordered abelian group, and groups acting on such spaces. The present paper concentrates on the case $\Lambda = \Z^n$ taken with the right lexicographic order and studies the structure of finitely generated groups acting on hyperbolic $\Z^n$-metric spaces. Under certain constraints, the structure of such groups is described in terms of a {\em hierarchy} (see \cite{Wise:2011}) similar to the one established for $\Z^n$-free groups in \cite{KMRS:2012}.
\end{abstract}

\maketitle

\tableofcontents

\section{Introduction}
\label{sec:intro}

In his paper \cite{Lyndon:1963}, Lyndon studied groups, where many standard cancellation techniques from the theory of free groups could be successfully applied. This is how groups with Lyndon length functions were introduced. Then in \cite{Chiswell:1976}, Chiswell showed that a group with a Lyndon length function taking values in $\R$ (or $\Z$) has an isometric action on an $\R$-tree (or $\Z$-tree), providing a construction of the tree on which the group acts. This result was later generalized by Alperin and Bass in \cite{Alperin_Bass:1987} to the case of an arbitrary ordered abelian group $\Lambda$. Hence, one can study groups with abstract Lyndon length functions taking values in $\Lambda$ by considering corresponding actions on $\Lambda$-trees, the objects introduced by Morgan and Shalen in \cite{Morgan_Shalen:1984}. One can think of $\Lambda$-trees as $0$-hyperbolic metric spaces, where the metric takes values not in $\R$, but rather in the ordered abelian group $\Lambda$.

Once the equivalence between $\Lambda$-valued Lyndon length functions and actions on $\Lambda$-trees is established, one can think of possible generalizations. For example what happens if a group acts on a hyperbolic $\Lambda$-metric space? Is there an underlying length function with values in $\Lambda$ in this case? This question was positively answered in \cite{GKMS:2012}, where the authors introduced hyperbolic $\Lambda$-valued length functions and studied many properties of groups that admit such functions.

Now let us concentrate on the case when $\Lambda = \Z^n$ equipped with the right lexicographic order. Our goal is to obtain structural properties of groups with hyperbolic $\Z^n$-valued length functions (that is, acting on hyperbolic $\Z^n$-metric spaces) in terms of free constructions. In this paper we made the first step in this direction: under several natural restrictions on the group and the underlying hyperbolic $\Z^n$-valued length function, we obtained a result similar to the description of groups with free regular Lyndon length functions with values in $\Z^n$ established in \cite{KMRS:2012}. In the future we hope to generalize our results by dropping some of the imposed conditions.

\smallskip

In order to be able to precisely formulate our results below, we need several definitions (see Section \ref{sec:prelim} for details).

Let $\Lambda$ be an ordered abelian group. A {\em $\Lambda$-metric space} is a space defined by the same set of axioms as a usual metric space but with $\R$ replaced by $\Lambda$. Now, if $(X, d)$ is a $\Lambda$-metric space, $v \in X$ and $\delta \in \Lambda$ is positive, then we call $(X, d)$ {\em $\delta$-hyperbolic with respect to $v$} if for all $x, y, z \in X$ 
$$(x \cdot y)_v \geqslant \min\{(x \cdot z)_v, (z \cdot y)_v\} - \delta,$$
where $(x \cdot y)_v$ denotes the Gromov product of $x$ and $y$ with respect to $v$. In fact, hyperbolicity does not depend on the choice of the point $v$, that is, a $\delta$-hyperbolic space with respect to one point is $2\delta$-hyperbolic with respect to any other point. So, we call $(X, d)$ {\em $\delta$-hyperbolic} (or simply {\em hyperbolic}) if it is $\delta$-hyperbolic with respect to every $v \in X$.

Now suppose a group $G$ is acting on a $\Lambda$-metric space $(X, d)$. If we fix $x_0 \in X$ then the action $G \curvearrowright X$ defines a function $l : G \rightarrow \Lambda$ by setting $l(g) = d(x, g x_0)$ and $l$ satisfies the standard axioms of length function (but not a Lyndon length function, in general). If $(X, d)$ is hyperbolic, then the length function $l$ defined above is also called {\em hyperbolic} and it satisfies the set of axioms ($\Lambda 1$)--($\Lambda 4$) (see Section \ref{sec:prelim} for details). Note that one can consider an abstract function $l : G \rightarrow \Lambda$ on $G$ satisfying the axioms ($\Lambda 1$)--($\Lambda 4$), that is, without underlying action on any space. In this case we still call such a function hyperbolic. 

Next, that $l$ is an abstract hyperbolic length function on $G$, we write $a b = a \circ b$ to signify that $l(ab) = l(a) + l(b)$ and we say that the function $l$ is {\em $\delta$-regular} if for any $g, h \in G$ there exist $g_c, h_c, g_d, h_d$ such that 
$$l(g_c) = l(h_c) = c(g,h),\ g = g_c \circ g_d,\ h = h_c \circ h_d,\ {\rm and}\ l(g_c^{-1} h_c) \leqslant 4\delta,$$
where
$$c(g, h) = \frac{1}{2}\left(l(g) + l(h) - l(g^{-1} h)\right).$$
Here we assume that $c(g, h) \in \Lambda$ (and, in general, that $\frac{1}{2} \alpha \in \Lambda$ for every $\alpha \in \Lambda$), since $\Lambda$ can always be replaced by $\Lambda_{\Q} = \Lambda \otimes_{\Z} \Q$, which is a divisible ordered abelian group containing $\Lambda$ as a subgroup. The notion of regularity of length function (for groups acting on $\Lambda$-trees) first appeared in \cite{Promislow:1985} and then was developed in \cite{Myasnikov_Remeslennikov_Serbin:2005}. Regularity of $\Lambda$-valued hyperbolic length functions was studied in \cite{GKMS:2012} (see Section 5.1).

\smallskip

Now consider the case $\Lambda = \Z^n$, where $\Z^n$ has the right lexicographic order. Such an order gives rise to the following notion of {\em height} (see \cite{Khan_Myasnikov_Serbin:2007}): for $a \in \Z^n$ we define the height $ht(a)$ of $a$ to be equal to $k$ if $a = (a_1, \ldots, a_k, 0, \ldots, 0)$ and $a_k \neq 0$. In particular, for $a, b > 0$, the inequality $ht(a) < ht(b)$ implies $a < b$.

Let $G$ be a group and let $l : G \to \Z^n$ be an abstract length function (that is, $l$ satisfies the axioms ($\Lambda 1$)--($\Lambda 4$)). From now on, for $g \in G$ we write $ht(g) = k$ if $ht(l(g)) = k$. Now, for any $k \in [1,n]$ define
$$G_k = \{g \in G \mid ht(g) \leqslant k\}.$$
It is easy to see that for any $f, g \in G$:
\begin{enumerate}
\item $ht(f g) \leqslant \max\{ht(f), ht(g)\}$,
\item $ht(g) = ht(g^{-1})$.
\end{enumerate}
Hence, every $G_k$ is a subgroup of $G$ and
$$1 = G_0 \leqslant G_1 \leqslant \cdots \leqslant G_n = G.$$
This chain of subgroups of $G$ will have a central role in our main result.

We say that the length function $l$ is {\em proper} if for any $k \in \N$ the set $\{g \mid l(g) \leqslant (k, 0, \ldots, 0)\}$ is finite. That is, the set of elements of $G$ whose lengths are bounded by any natural number (considered as a subset of $\Z^n$) is finite if the length function is proper. For example, if $l : G \to \Z$ is the word length function associated with a finite generating set of $G$ then $l$ is obviously proper. But if $l$ is arbitrary and it takes values in $\Z^n$ for $n > 1$, then it may well be the case that there are infinitely many group elements whose lengths are bounded by $k \in \N$.

\smallskip

Now we are ready to formulate the main result of the paper. Recall that a subgroup $H$ of a group $G$ is called {\em isolated} if whenever there exists some $n \in \Z$ such that $g^n \in H$, it follows that $g \in H$.

{\bf Theorem \ref{th:main_result}.} {\em Let $G$ be a finitely generated torsion-free group with a proper $\delta$-hyperbolic $\delta$-regular length function $l : G \rightarrow \Z^n$, where $\Z^n$ has the right lexicographic order. Let $G_k = \{g \in G \mid ht(g) \leqslant k\}$ so that $G$ is the union of the finite chain of subgroups
$$1 = G_0 \leqslant G_1 \leqslant \cdots \leqslant G_n = G.$$ 
If the following conditions are satisfied
\begin{enumerate}
\item[(a)] $ht(\delta) = 1$,
\item[(b)] $l(g) > 0$ for every non-trivial $g \in G$,
\item[(c)] $G_1$ is a finitely generated isolated subgroup of $G$,
\end{enumerate}
then $G_1$ is a word-hyperbolic group and for every $k \in [1, n)$, the group $G_{k+1}$ is isomorphic to an HNN extension of $G_k$ with a finite number of stable letters, where each associated subgroup is isomorphic to either $\Z^m$, or an extension of $\Z^m$ by $\Z$, where $m \leqslant k - 1$.}

\smallskip

Note that the condition that $G$ is finitely generated is required to make sure that $G_1$ is word hyperbolic and it is sufficient in this context. There are, of course, non-finitely generated groups with proper regular actions on hyperbolic spaces, in which case Theorem \ref{th:main_result} may not hold. For example, consider an infinitely generated free group $F = F(X)$, where $X = \{x_i \mid i \in \N\}$, acting on its Cayley graph $\Gamma(F, X)$. If all edges of $\Gamma(F, X)$ have the same length $1$ then the length function arising from the action is regular but not proper. At the same time, if we define the length of every edge labeled by $x_n$ to be $n$ then regularity still holds, but the length function becomes proper. Observe that $F(X)$ cannot be decomposed as shown in Theorem \ref{th:main_result}: $G_1$ in this case coincides with $F(X)$ itself which is not hyperbolic.

\smallskip

Most of the results of this paper also appeared in the first author's Ph.D Thesis (see \cite{Grecianu:2013}).

\medskip

\noindent {\bf Acknowledgment.} The authors would like to thank Olga Kharlampovich and Andrey Nikolaev for insightful discussions. We would also like to thank the referees for their thorough reviews and very useful comments, which helped immensely in improving the clarity and relevance of the paper.

\section{Preliminaries}
\label{sec:prelim}

Here we recall the basic definitions regarding hyperbolic $\Lambda$-metric spaces (see \cite{GKMS:2012} for details).

\subsection{$\Lambda$-hyperbolic spaces and length functions}
\label{subs:hyp_spaces}

Recall that an {\em ordered} abelian group is an abelian group $\Lambda$ (with addition denoted by ``$+$'') equipped with a linear order ``$\leqslant$''  such that for all $\alpha, \beta, \gamma \in \Lambda$ the inequality $\alpha \leqslant \beta$ implies $\alpha + \gamma \leqslant \beta + \gamma$.

Sometimes we would like to be able to divide elements of $\Lambda$ by non-zero integers. To this end we fix a canonical order-preserving embedding of $\Lambda$ into an ordered {\em divisible} abelian group $\Lambda_\Q$  and identify $\Lambda$ with its image in $\Lambda_\Q$. The group $\Lambda_\Q$ is the tensor product $\Q \otimes_\Z \Lambda$ of two abelian groups (viewed as $\Z$-modules) over $\Z$. One can represent elements of $\Lambda_\Q$ by fractions $\tfrac{\lambda}{m}$, where $\lambda \in \Lambda, m \in \Z, m \neq 0$,  and two fractions $\tfrac{\lambda}{m}$ and $\tfrac{\mu}{n}$ are equal if and only if $n \lambda = m \mu$. Addition of fractions is defined as usual, and the embedding is given by the map $\lambda \to \tfrac{\lambda}{1}$. The order on $\Lambda_\Q$ is defined by $\tfrac{\lambda}{m} \geqslant 0 \Longleftrightarrow  m \lambda \geqslant 0 \text{ in } \Lambda$.  Obviously, the embedding $\Lambda \to \Lambda_\Q$ preserves the order. It is easy to see that $\R_\Q = \R$ and $\Z_\Q = \Q$. Furthermore, it is not hard to  show that $(A \oplus B)_\Q \simeq A_\Q \oplus B_\Q$, so $(\R^n)_\Q = \R^n$ and $(\Z^n)_\Q = \Q^n$.

For elements $\alpha, \beta \in \Lambda$ the {\em closed segment} $[\alpha, \beta]$ is defined by
$$[\alpha, \beta] = \{\gamma \in \Lambda \mid \alpha \leqslant \gamma \leqslant \beta \}.$$
Now, a  subset $C \subset \Lambda$ is called {\em convex} if for every $\alpha, \beta \in C$, the set $C$ contains $[\alpha, \beta]$. In particular, a subgroup $C$ of $\Lambda$ is convex if $[0, \beta] \subset C$ for every positive $\beta \in C$.

\smallskip

Let $X$ be a non-empty set and $\Lambda$ an ordered abelian group. A {\em $\Lambda$-metric} on $X$ is a mapping $d: X \times X \rightarrow \Lambda$ such that:
\begin{enumerate}
\item[(LM1)] $\forall\ x,y \in X:\ d(x,y) \geqslant 0$;
\item[(LM2)] $\forall\ x,y \in X:\ d(x,y) = 0 \Leftrightarrow x = y$;
\item[(LM3)] $\forall\ x,y \in X:\ d(x,y) = d(y,x)$;
\item[(LM4)] $\forall\ x,y,z \in X:\ d(x,y) \leqslant d(x,z) + d(y,z)$.
\end{enumerate}

A {\em $\Lambda$-metric space} is a pair $(X,d)$, where $X$ is a non-empty set and $d$ is a $\Lambda$-metric on $X$. If $(X,d)$ and $(X',d')$ are $\Lambda$-metric spaces, an {\em isometry} from $(X,d)$ to $(X',d')$ is a mapping $f: X \rightarrow X'$ such that $d(x,y) = d'(f(x), f(y))$ for all $x, y \in X$. As in the case of usual metric spaces, a {\em segment} in a $\Lambda$-metric space $X$ is the image of an isometry $\alpha: [a,b] \rightarrow X$ for some $a,b \in \Lambda$. In this case  $\alpha(a),  \alpha(b)$ are called the endpoints of the segment. By $[x,y]$ we denote any segment with endpoints $x,y \in X$.

We call a $\Lambda$-metric space $(X,d)$ {\em geodesic} if for all $x,y \in X$, there is a segment in $X$ with endpoints $x,y$.

Let $(X,d)$ be a $\Lambda$-metric space. Fix a point $v \in X$  and for $x,y \in X$ define the Gromov product
$$(x \cdot y)_v = \frac{1}{2} (d(x,v) + d(y,v) - d(x,y)),$$
as an element of $\Lambda_\Q$. Obviously, this is the direct generalization of the standard Gromov product to the $\Lambda$-metric case. Most of its classical properties remain true in the generalization (see \cite[Section 2.3]{GKMS:2012}).

Let $\delta \in \Lambda$ with $\delta \geqslant 0$. Then $(X, d)$ is {\em $\delta$-hyperbolic with respect to $v$} if, for all $x, y, z \in X$
$$(x \cdot y)_v \geqslant \min\{(x \cdot z)_v, (z \cdot y)_v\} - \delta.$$

It was proved in \cite[Lemma 1.2.5]{Chiswell:2001} that if $X$ is $\delta$-hyperbolic with respect to $x$ then it is $2\delta$-hyperbolic with respect to any other point $y \in X$. In view of this result, we call a $\Lambda$-metric space $(X, d)$ {\em $\delta$-hyperbolic} if it is $\delta$-hyperbolic with respect to any point.

One of the crucial examples of $\delta$-hyperbolic $\Lambda$-metric spaces is {\em $\Lambda$-tree} which is exactly a $0$-hyperbolic geodesic $\Lambda$-metric space $X$ such that the Gromov product of any two points $x,y \in X$ with respect to any other point $v \in X$ is always an element of $\Lambda$, that is, $(x\cdot y)_v \in \Lambda$ for all $a, y, v \in X$.

\smallskip

In \cite{Lyndon:1963} Lyndon introduced a notion of an (abstract) {\em length function} $l : G \to \Lambda$ on a group $G$ with values in 
$\Lambda$. Such a function $l$ satisfies the following axioms:
\begin{enumerate}
\item[($\Lambda 1$)] $\forall\ g \in G:\ l(g) \geqslant 0$ and $l(1) = 0$,
\item[($\Lambda 2$)] $\forall\ g \in G:\ l(g) = l(g^{-1})$,
\item[($\Lambda 3$)] $\forall\ g,h \in G:\ l(g h) \leqslant l(g) + l(h)$.
\end{enumerate}

Again, the easiest example of a $\Z$-valued length function on any group $G$ with a generating set $S$ is the word length $|\, \cdot\, |_S : G \to \Z$, where $|g|_S$ is the minimum length of a word $w(S)$ representing the element $g \in G$.

We can introduce another axiom which is a direct generalization of the axiom Lyndon introduced. A length function $l : G \to \Lambda$ is called {\em hyperbolic} if there is $\delta \in \Lambda$ such that

\smallskip

{\noindent ($\Lambda 4, \delta$) $\forall\ f, g, h \in G:\ c(f,g) \geqslant \min\{c(f,h),c(g,h)\} - \delta$,

\smallskip

where $c(g,h) = \frac{1}{2}\left(l(g) + l(h) - l(g^{-1} h)\right)$ is viewed as an element of $\Lambda_\Q$.

If $G$ acts on a $\Lambda$-metric space $(X, d)$ then one can fix a point $v \in X$ and consider a function $l_v : G \to \Lambda$ defined as $l_v(g) = d(v, g v)$, called  a {\em length function based at $v$}. It is not hard to show that $l_v$ satisfies the axioms ($\Lambda 1$)--($\Lambda 3$), that is, it is a length function on $G$ with values in $\Lambda$. Moreover, if $(X, d)$ is $\delta$-hyperbolic for some $\delta \in \Lambda$ with respect to $v$, then $l_v$ is $\delta$-hyperbolic (see 
\cite[Theorem 3.1]{GKMS:2012}).

It is easy to see that a length function $l : G \to \Lambda$ defines a $\Lambda$-pseudometric $d_l $ on $G$ defined by $d_l(g, h) = l(g^{-1} h)$. In the case when $l(g) = 0$ if and only if $g$ is trivial, the pseudometric $d_l$ becomes a $\Lambda$-metric.

\smallskip

Now suppose that $l : G \to \Lambda$ is a $\delta$-hyperbolic length function. We write $a b = a \circ b$ to signify that $l(ab) = l(a) + l(b)$ and we say that the function $l$ is {\em $\delta$-regular (on $G$)} if for any $g, h \in G$ there exist $g_c, h_c, g_d, h_d \in G$ such that 
$$l(g_c) = l(h_c) = c(g,h),\ g = g_c \circ g_d,\ h = h_c \circ h_d,\ {\rm and}\ l(g_c^{-1} h_c) \leqslant 4\delta.$$
Compare this definition with the ones using the properties ($R1, k$), ($R2, k$), and ($R3, k$) introduced in \cite[Section 5.1]{GKMS:2012}.

The regularity property of (Lyndon) length functions was first considered in \cite{Promislow:1985}. Then it was developed in \cite{Myasnikov_Remeslennikov_Serbin:2005} for groups acting freely on
$\Lambda$-trees. It turns out that this property makes many combinatorial arguments based on cancellation between elements very similar to the case of free groups (for example, Nielsen method in free groups). In particular, it is possible to describe groups which have regular free Lyndon length functions in $\Z^n$ (see \cite{KMRS:2012}) and, more generally, in an arbitrary ordered abelian group $\Lambda$ (see \cite{KMS:2011(3)}, \cite{KMS_Survey:2013}). Regularity of $\Lambda$-valued hyperbolic length functions was first introduced in \cite{GKMS:2012}, where basic properties of such functions were studied.

In our case, regularity of a $\delta$-hyperbolic length function on a group $G$ will help us determine the structure of $G$ similarly to the case $\delta = 0$.

\subsection{The case $\Lambda = \Z^n$}
\label{subs:Z^n}

Now consider the case when $\Lambda = \Z^n$, where $\Z^n$ has the right lexicographic order. Recall that if $A$ and $B$ are ordered abelian groups, then the {\em right lexicographic order} on the direct sum $A \oplus B$ is defined as follows:
$$(a_1,b_1) < (a_2,b_2) \Leftrightarrow b_1 < b_2 \ \mbox{or} \ b_1 = b_2 \ \mbox{and} \ a_1 < a_2.$$
One can easily extend this definition to any number of components in the direct sum and apply it in the case of $\Z^n$ which is the direct sum of $n$ copies of $\Z$.

Every element $a \in \Z^n$ can be represented by an $n$-tuple $(a_1, \ldots, a_k, 0, \ldots, 0)$. We say that the {\em height} of $a$ is equal to $k$, and write $ht(a) = k$, if $a = (a_1, \ldots,$ $a_k, 0, \ldots, 0)$ and $a_k \neq 0$. The notion of height can be defined for an arbitrary ordered abelian group $\Lambda$ since $\Lambda$ can be represented as the union of all its convex subgroups and the height of $\alpha \in \Lambda$ is the smallest index of the convex subgroup of $\Lambda$ that $\alpha$ belongs to.

Speaking of convex subgroups of $\Z^n$, each one of them is of the following type
$$\{a \in \Z^n \mid ht(a) \leqslant k\},$$
which is naturally isomorphic to $\Z^k$. Hence, we have a (finite) complete chain of convex subgroups of $\Z^n$:
$$0 < \Z < \Z^2 < \cdots < \Z^n.$$
Now, if $(X, d)$ is a $\Z^n$-metric space and $C \leqslant \Z^n$ is a convex subgroup, then we can consider the subset
$$X_{x, C} = \{y \in X \mid d(x, y) \in C\},$$
which is in turn a $C$-metric space with respect to the metric $d_0 = d|_{X_{x, C}}$. We call $(X_{x, C}, d_0)$ a {\em $C$-metric subspace} of $X$. In view of the complete chain of convex subgroups of $\Z^n$ above, we may have $\Z^k$-subspaces of $X$ for every $k \in [1, n]$. In particular, we can say that a subspace $X_0$ of $X$ {\em has height $k$} if $X_0 = X_{x, C}$ for some $x \in X$ and $C = \Z^k$.

\smallskip

Let $G$ act on a $\Z^n$-metric space $(X, d)$. We say that the action is {\em proper} if for any $x \in X$, the intersection $(G \cdot x) \cap B_\kappa(x)$ is finite for every $\kappa = (a, 0, \ldots, 0)$, where $a \in \N$ (here $G \cdot x$ is the orbit of the point $x \in X$ and $B_\kappa(x)$ is a ball of radius $\kappa$ in $X$ centered at $x$). We say that the action is {\em co-compact} if there exists $\kappa = (a, 0, \ldots, 0)$, where $a \in \N$, such that for every $x, y \in X$ there exist $x' \in G \cdot x,\ y' \in G \cdot y$ with the property $d(x', y') \leqslant \kappa$.

Now, assume that a group $G$ has a $\delta$-hyperbolic length function $l$ with values in $\Z^n$. In many cases we will refer to the height $ht(g)$ of $g \in G$, which is simply the height of its length $ht(l(g))$.

One of the first examples of $\Z^n$-valued hyperbolic length functions is the function $l : G \to \Z$, where $G$ is a word-hyperbolic group with a generating set $S$ and $l = |\, \cdot\, |_S$ is the word length with respect to $S$. There are more sophisticated examples.

\begin{example}
\label{ex:1}
Let $G$ be a group acting on a hyperbolic $\Z$-metric space $X$ and a $\Z$-tree $Y$. Let $x \in X$ and $y \in Y$. 

From the actions of $G$ on $X$ and $Y$ we obtain the length functions $l_X : G \to \Z$ and $l_Y : G \to \Z$ based respectively at $x$ and $y$, that is, defined as $l_X(g) = d(x, g x)$ and $l_Y(g) = d(y, g y)$. Next, we have
$$c_X(g, h) = \frac{1}{2}\left(l_X(g) + l_X(h) - l_X(g^{-1} h)\right)$$
and 
$$c_Y(g, h) = \frac{1}{2}\left(l_Y(g) + l_Y(h) - l_Y(g^{-1} h)\right).$$

Define $l : G \rightarrow \Z^2$ as $l = (l_X, l_Y)$, that is, 
$$l(g) = (d(x, g x), d(y, g y)).$$  
Hence, we have
$$c(g,h) = \frac{1}{2}\left(l(g) + l(h) -l(g^{-1} h)\right) = (c_X(g, h), c_Y(g, h)).$$
This implies that, for any $g, h, k \in G$, we have
$$c(g, k) = \left(c_X(g, k), c_Y(g, k)\right)$$
$$\geqslant \left(\max\{c_X(g, h), c_X(h, k)\} - \delta, \max\{c_Y(g, h), c_Y(h, k)\}\right)$$
$$= \max\{c(g, k), c(h, k)\} - (\delta, 0).$$
This implies that any group that acts both on a simplicial tree, and a hyperbolic $\Z$-metric space, has a $\Z^2$-valued hyperbolic length function. The same construction is valid for $Y$ being a $\Z^n$-tree.
\end{example}

Let $l : G \to \Z^n$ be $\delta$-hyperbolic. We say that the length function $l$ is {\em proper} if for any $k \in \N$ the set $\{g \mid l(g) \leqslant (k, 0, \ldots, 0)\}$ is finite. Observe that if $l$ is a based length function coming from a proper action of $G$ on some $\Z^n$-metric space $X$, then $l$ itself is proper. A proper length function corresponds to an action that is properly discontinuous on subspaces of height $1$, hence the term.


\begin{example}
\label{ex:2}
Let a finitely generated group $G$ act properly and co-compactly on a geodesic $\Z^n$-hyperbolic space $X$ so that the stabilizer of some $\Z$-subspace $X_0$ is finitely generated.

Since $G$ acts co-compactly, $X / G$ has finite diameter of minimal height, which means that there exists $a > 0$ such that for any $x,y \in X$, there exists $u \in G$ with the property $u x \in B_\kappa(y)$, where $\kappa = (a, 0, \ldots, 0)$. This naturally implies that the action of $G$ satisfies the condition ($RA, k$) (see \cite[Section 5.1]{GKMS:2012}) for a natural number $k$ such that $a < k \delta$, since for every $g, h \in G$, there is always a mid-point $y$ of the geodesic triangle $\{x, gx, hx\}$ and there exists some $u \in G$ such that $u x \in B_\kappa(y)$.

Take then $x_0 \in X_0$ and define $l(g) = d(x_0, g x_0)$. Note that $l$ is hyperbolic since $X$ is hyperbolic. Moreover, it is regular which follows from the fact that the action satisfies the condition ($RA, k$) (see \cite[Lemma 5.3]{GKMS:2012}).

It follows that natural conditions (proper and co-compact action) on a geodesic $\Z^n$-hyperbolic space ensure existence of a regular and proper $\Z^n$-valued length function.
\end{example}

Finally, whenever we deal with a $\Z^n$-tree, which is a $0$-hyperbolic $\Z^n$-metric space, we are going to use the following notation. Let $(\Gamma, d)$ be a $\Z^n$-tree with a metric $d$. The set of points (or vertices) of $\Gamma$ we denote by $V(\Gamma)$, then for every $v_0, v_1 \in V(\Gamma)$ such that $d(v_0, v_1) = 1$, we call the ordered pair $(v_0, v_1)$ the {\em edge} from $v_0$ to $v_1$. Here, if $e = (v_0, v_1)$ then denote $v_0 = o(e),\ v_1 = t(e)$ which are respectively the {\em origin} and {\em terminus} of $e$. Denote by $E(\Gamma)$ the set of edges in $\Gamma$. Next, we define the {\em inversion} function $^{-1} : E(\Gamma) \to E(\Gamma)$: for an edge $e \in E(\Gamma)$ the endpoints of the {\em inverse} $e^{-1}$ of $e$ we define as $o(e^{-1}) = t(e),\ t(e^{-1}) = o(e)$, that is, the direction of $e^{-1}$ is reversed with respect to the direction of $e$. It is easy to see that we have $(e^{-1})^{-1} = e$ and since $\Gamma$ has no loops, $e \neq e^{-1}$ for any $e \in E(\Gamma)$. 

The notation above is quite standard and we are going to use it throughout the paper.

\section{The structure of groups with $\Z^n$-valued hyperbolic length functions}
\label{sec:structure}

In this section we prove the main result of the paper.

\subsection{Auxiliary result}
\label{subs:main_theorem}

In order to prove Theorem \ref{th:main_result} we need an auxiliary result with a similar statement, that we are going to formulate in the end of this subsection. Throughout this subsection we assume that $G$ is a torsion-free group with a $\delta$-hyperbolic $\delta$-regular length function $l : G \rightarrow \Z^n$, where $\Z^n$ has the right lexicographic order. Moreover, we assume that the following conditions are satisfied:
\begin{enumerate}
\item[(a)] $ht(\delta) = 1$,
\item[(b)] $l(g) > 0$ for every non-trivial $g \in G$,
\item[(c)] $G_1$ is a finitely generated isolated subgroup of $G$.
\end{enumerate}

Recall that $l(g) = (a_1, \ldots, a_n)$ for every $g \in G$. Hence, for every $k \in [0, n)$, let $l_k(g)$ be the restriction of $l$ to the $n - k$ rightmost coordinates, that is, $l_k(g) = (a_{k + 1}, \ldots, a_n)$ (in particular, $l_0 = l$). In other words, $l_k = \pi_k \circ l$, where $\pi_k$ is the projection of $n$-tuple on the $n - k$ rightmost coordinates, which is an order-preserving homomorphism of $\Z^n$ onto $\Z^{n - k}$. Note that $l_k$ is a $\Z^{n - k}$-valued length function on $G$ (the axioms ($\Lambda 1$)--($\Lambda 3$) are trivially satisfied). Observe also that if $X$ is a hyperbolic $\Z^n$-metric space corresponding to the length function $l$ (see \cite[Theorem 3.3]{GKMS:2012}), then $l_k$ corresponds to the action of $G$ on the $\Z^{n - k}$-metric space $Y$ obtained from $X$ by collapsing all $\Z^k$-subspaces of $X$ into points.
Similar to the notation
$$c(g, h) = \frac{1}{2}\left(l(g) + l(h) - l(g^{-1} h)\right),$$
for every $f, g \in G$ we define
$$c_k(g, h) = \frac{1}{2}\left(l_k(g) + l_k(h) - l_k(g^{-1} h)\right).$$
Next, for every $m \in [0, n]$ define $G_m = \{g \in G \mid ht(g) \leqslant m\}$. It is easy to see that $G_m \leqslant G_{m+1}$ and we have the following chain of subgroups of $G$
$$1 = G_0 \leqslant G_1 \leqslant \cdots \leqslant G_n = G,$$
which corresponds to the complete chain of convex subgroups of $\Z^n$. Since $G_m$ is a subgroup of $G$, the restriction of $l_k$ on $G_m$ is a length function for every $k \in [0, n)$ (in particular, $l_k : G_m \to \Z^{n-k}$ is a zero function if $m \leqslant k$). We list more specific properties of $l_k$ in the lemma below.

\begin{lemma}
\label{lem:0-hyp}
For any $k \in (0, n)$, the length function $l_k : G \to \Z^{n - k}$ satisfies the following conditions:
\begin{itemize}
\item $l_k$ is $0$-hyperbolic,
\item $l_k$ is a Lyndon length function on $G$,
\item for every $m \in [0, n]$, the restriction of $l_k$ on $G_m$ is $0$-regular on $G_m$, that is, for every $g, h \in G_m$ there exist $g_c, h_c, g_d, h_d \in G_m$ such that 
$$l_k(g_c) = l_k(h_c) = c_k(g,h),\ g = g_c \circ g_d,\ h = h_c \circ h_d,\ {\rm and}\ l_k(g_c^{-1} h_c) = 0.$$
\end{itemize}
\end{lemma}
\begin{proof}
Notice that $c_k = \pi_k \circ c$, where $\pi_k$ is the projection of $n$-tuple on the $n - k$ rightmost coordinates. Since $\pi_k$ is an order-preserving homomorphism of $\Z^n$ onto $\Z^{n-k}$, from the inequality
$$c(g, h) \geqslant \min\{c(g, f), c(h, f)\} - \delta$$ 
for any $f,\ g$, and $h$, we obtain that 
$$c_k(g, h) = \pi_k(c(g, h)) \geqslant \min\{\pi_k(c(g, f)),\pi_k(c(h, f))\} - \pi_k(\delta)$$
$$= \min\{c_k(g, f), c_k(h, f)\},$$
since $\pi_k(\delta) = 0$ for any $k > 0$ (it follows from the assumption (a) that $ht(\delta) = 1$). Hence, $l_k$ is $0$-hyperbolic, which implies that $l_k$ is a Lyndon length function on $G$.

\smallskip

Finally, we show that the restriction of $l_k$ on $G_m$ is $0$-regular on $G_m$. By assumption, $l$ is $\delta$-regular on $G$, that is, for any $g, h \in G$ there exist $g_c, h_c, g_d, h_d \in G$ such that $g = g_c g_d,\ h = h_c h_d$ and
$$l(g) = l(g_c) + l(g_d),\ \ \ l(h) = l(h_c) + l(h_d),$$
$$l(g_c) = l(h_c) = c(g,h),\ \ \ l(g_c^{-1} h_c) \leqslant 4\delta.$$
After $\pi_k$ is applied, we obtain
$$l_k(g) = l_k(g_c) + l_k(g_d),\ \ \ l_k(h) = l_k(h_c) + l_k(h_d),$$
$$l_k(g_c) = l_k(h_c) = c_k(g,h),\ \ \ l_k(g_c^{-1} h_c) = 0,$$
and it is only left to show that if $g, h \in G_m$, then $g_c, h_c, g_d, h_d \in G_m$. But it is easy to see that for $f_1, f_2 \in G$, if $l(f_1 f_2) = l(f_1) + l(f_2)$, then $ht(f_1 f_2) = \max\{ ht(f_1), ht(f_2) \}$. Thus, from $f_1 f_2 \in G_m$ we get  $ht(f_1 f_2) \leqslant m$, which implies that $ht(f_1), ht(f_2) \leqslant m$. Hence, $g, h \in G_m$ implies that $g_c, h_c, g_d, h_d \in G_m$.
\end{proof}

By Lemma \ref{lem:0-hyp}, for every $k \in (0, n)$ and $m \in (k, n]$, the function $l_k : G_m \to \Z^{m - k}$ is a Lyndon length function on $G_m$, hence, by \cite[Theorem 4.6]{Chiswell:2001}, there exists a $\Z^{m - k}$-tree $(\Gamma^m_k, d^m_k)$, an action of $G_m$ on $\Gamma^m_k$, and a point $x^m_k \in \Gamma^m_k$ such that $l_k(g) = d^m_k(x^m_k, g x^m_k)$ for every $g \in G_m$. 

Let us fix $k \in (0, n)$ and $m \in (k, n]$. Let us label some vertices of $\Gamma^m_k$ as follows. Observe that for $f \in G_m$ we have $l_k(f) = 0$ if and only if $f \in G_k$. Indeed, since $ht(f) \leqslant m$, we have $l(f) = (a_1, \ldots, a_k, \ldots, a_m, 0, \ldots, 0)$. Now, $l_k(f) = 0$ (that is, $l_k(f)$ is an $(m - k)$-tuple of zeros) is equivalent to $a_{k+1} = \cdots = a_m = 0$, which is equivalent to $f \in G_k$. It follows that $f x^m_k = x^m_k$ for every $f \in G_k$ and $g x^m_k = h x^m_k$ if $g, h \in G_m$ and $g G_k = h G_k$. Hence, for every $g \in G_m$, label the vertex $g x^m_k$ by the left coset $g G_k$ of $g$. In this case, the action of $G_m$ on labeled vertices of $\Gamma^m_k$ is equivalent to the action of $G_m$ on the set of left cosets of $G_k$ in $G_m$ by left multiplication: $h (g x^m_k) = (h g) x^m_k$ if and only if $h (g G_k) = (h g) G_k$.

Notice that by Lemma \ref{lem:0-hyp}, $l_k$ is $0$-regular on $G_m$, that is, for every $g, h \in G_m$ there exist $g_c, h_c, g_d, h_d \in G_m$ such that 
$$l_k(g_c) = l_k(h_c) = c_k(g, h),\ g = g_c \circ g_d,\ h = h_c \circ h_d,\ {\rm and}\ l_k(g_c^{-1} h_c) = 0.$$
It follows that $g_c^{-1} h_c \in G_k$ and $g_c x^m_k = h_c x^m_k$. Notice that the vertex $g_c x^m_k$ labeled by the coset $g_c G_k$ is a branch point, which follows from the equality $l_k(g_c) = l_k(h_c) = c_k(g, h)$. Moreover, every branch point of $\Gamma^m_k$ is labeled by some coset of $G_k$ in $G_m$. Indeed, if $x$ is a branch point, then there exist $g, h \in G_m$ such that $x = Y(x^m_k, g x^m_k, h x^m_k)$ (that is, $[x^m_k, x] = [x^m_k, g x^m_k] \cap [x^m_k, h x^m_k]$). In this case, from $0$-regularity of $l_k$ it follows that there exists $u \in G_m$ such that $x = u x^m_k$ and $x$ is labeled by $u G_k$. In particular, it follows that the degree of every non-labeled vertex of $\Gamma^m_k$ is at most $2$.

\smallskip

Now consider the $\Z^{n-1}$-tree $\Gamma_1^n$, which $G = G_n$ acts upon with the underlying length function $l_1 : G \to \Z^{n - 1}$. Denote the metric on $\Gamma_1^n$ by $d_1$.

Recall (see \cite[Section 3.1]{Chiswell:2001}) that the {\em axis} (or the {\em characteristic set}) of $g \in G$ is the subset $Axis(g)$ of $\Gamma_1^n$ defined by
$$Axis(g) = \{ p \in \Gamma_1^n \mid [g^{-1} p, p] \cap [p, g p] = \{p\} \}.$$
If $g$ acts on $\Gamma_1^n$ as an inversion (that is, $g^2$ has a fixed point, but $g$ does not), then $Axis(g) = \varnothing$. If $g$ acts either elliptically, or hyperbolically, then $Axis(g)$ is a closed non-empty $\langle g \rangle$-invariant subtree of $\Gamma_1^n$ and $Axis(g) = \{ p \in \Gamma_1^n \mid d_1(p, g p) = \ell(g) \}$, where $\ell(g) = \min\{ d_1(x, gx) \mid x \in \Gamma_1^n \}$ is the {\em translation length} of $g$ (see \cite[Corollary 3.1.5]{Chiswell:2001}). Notice that if $g$ acts elliptically, then $\ell(g) = 0$ and if $g$ acts hyperbolically, then $\ell(g) > 0$, in which case for every $m \neq 0$ we have $Axis(g^m) = Axis(g)$ and $\ell(g^m) = |m| \ell(g)$ (see \cite[Corollary 3.1.7]{Chiswell:2001}). 

There is a connection between $l_1(g)$ and $\ell(g)$. If $g$ acts elliptically, then $l_1(g) = 2d_1(x_1^n, Axis(g))$ (see \cite[Lemma 3.1.1]{Chiswell:2001}). In this case, for every $m \neq 0$ we have $l_1(g^m) = 2d_1(x_1^n, Axis(g^m)) \leqslant 2d_1(x_1^n, Axis(g)) = l_1(g)$ since $Axis(g) \subseteq Axis(g^m)$. If $g$ acts hyperbolically, then  $l_1(g) = \ell(g) + 2d_1(x_1^n, Axis(g))$ (see \cite[Theorem 3.1.4]{Chiswell:2001}). It follows that for every $m \neq 0$ we have 
$$l_1(g^m) = \ell(g^m) + 2d_1(x_1^n, Axis(g^m)) = |m| \ell(g) + 2d_1(x_1^n, Axis(g))$$
$$> \ell(g) + 2d_1(x_1^n, Axis(g)) = l_1(g)$$
since $Axis(g^m) = Axis(g)$.

\smallskip

Next, we show that the condition (c) given above (and in the statement of Proposition \ref{pr:main_technical}, that follows in the end of the section and which is the main result of the section) is equivalent to a more technical condition, which will be easier to use later.

\begin{lemma}
\label{le:conditions_1}
Let $G$ be a torsion-free group with a $\delta$-hyperbolic $\delta$-regular length function $l : G \rightarrow \Z^n$. If the conditions (a) and (b) hold, then the condition (c) is equivalent to the following one:
\begin{enumerate}
\item[(c$'$)] for any $g \in G$ and $m \neq 0$, if $l(g^m) < l(g)$, then $ht(l(g) - l(g^m)) = 1$.
\end{enumerate}
Moreover, the conditions (a), (b), and (c) imply that any $g \in G$ acts on $\Gamma_1^n$ either elliptically, or hyperbolically, but not as an inversion.
\end{lemma}
\begin{proof}
$(c')\ \Longrightarrow\ (c)$: suppose $g \in G$ is such that $g^m \in G_1$ for some $m \neq 0$. If $l(g^m) \geqslant l(g)$, then $ht(l(g)) = 1$ and $g \in G_1$. Suppose that $l(g^m) < l(g)$. From (c$'$) it follows that $ht(l(g) - l(g^m)) = 1$. Since $ht(l(g^m)) = 1$, we obtain that $ht(l(g)) = 1$ and $g \in G_1$.

\smallskip

$(c)\ \Longrightarrow\ (c')$: let $g \in G$ and consider the action of $g$ on $\Gamma_1^n$.

If $g$ acts hyperbolically on $\Gamma_1^n$, then $l_1(g^m) > l_1(g)$ for every $m \neq 0$, which implies that $l(g^m) > l(g)$. Hence, (c$'$) is not applicable here.

If $g$ acts elliptically, then for every $m \neq 0$ we have $l_1(g^m) \leqslant l_1(g)$. If $Axis(g) = Axis(g^m)$, then 
$l_1(g) = l_1(g^m)$, which implies $l(g) = l(g^m) + \varepsilon$, where $ht(\varepsilon) = 1$, hence, $ht(l(g) - l(g^m)) = 1$. In particular, (c$'$) holds if $l(g^m) < l(g)$ for some $m$. Now suppose the fixed point set $Axis(g)$ of $g$ is properly contained in the fixed point set $Axis(g^m)$ of $g^m$, that is, there exists $v \in Axis(g^m)$ such that $g^m v = v$ but $g v \neq v$. Recall that some vertices of $\Gamma_1^n$ are labeled by cosets of $G_1$ in $G$ and non-labeled vertices have degree at most $2$. If $v$ is labeled, say by $a G_1$ for some $a \in G$, then $(g a) G_1 \neq a G_1$, but $(g^m a) G_1 = a G_1$, which implies that $(a^{-1} g a) G_1 \neq G_1$, but $(a^{-1} g^m a) G_1 = G_1$. Thus, $a^{-1} g a \notin G_1$, but $(a^{-1} g a)^m \in G_1$, which is a contradiction with (c). Hence, $v$ is not labeled. But in this case $v$ belongs to a geodesic segment between two labeled vertices, say $[a G_1, b G_1]$ for some $a, b \in G$. Since $g^m$ fixes $v$ and $g$ does not, $g^{2m}$ fixes $[a G_1, b G_1]$ and $g$ does not. In particular, $g^{2m}$ fixes the labeled vertex $a G_1$ which is not fixed by $g$. Repeating the argument above we obtain that $a^{-1} g a \notin G_1$, but $(a^{-1} g a)^{2m} \in G_1$, which is a contradiction again. Hence, the assumption that $Axis(g) \subsetneq Axis(g^m)$ leads to a contradiction.

Finally, suppose $g$ acts as an inversion. Then by definition there exists a vertex $v$ in $\Gamma_1^n$ such that $g^2 v = v$ and $g v \neq v$. Repeating the argument for the elliptic case we obtain a vertex in $\Gamma_1^n$ labeled by $a G_1$, which is fixed by $g^2$, but not by $g$, and further have $a^{-1} g a \notin G_1$, but $(a^{-1} g a)^2 \in G_1$, which is a contradiction with (c).

It follows that $l(g^m) < l(g)$ can occur only when $g \in G$ acts elliptically and in this case we have $ht(l(g) - l(g^m)) = 1$, that is, (c$'$) follows from (c).
\end{proof}

Now, using the above lemma, we can formulate another property of the length function $l_k : G \to \Z^{n - k}$, in addition to those listed in Lemma \ref{lem:0-hyp}.

\begin{corollary}
\label{cor:0-hyp}
Let $G$ be a torsion-free group with a $\delta$-hyperbolic $\delta$-regular length function $l : G \rightarrow \Z^n$. If the conditions (a), (b), and (c) hold, then for any $k \in (0, n)$, the length function $l_k : G \to \Z^{n - k}$ satisfies the property that $l_k(g^2) \geqslant l_k(g)$ for any $g \in G$.
\end{corollary}
\begin{proof}
From Lemma \ref{le:conditions_1} it follows that the conditions (a), (b), and (c) imply (c$'$). Hence, for any $g \in G$, either $l(g^2) \geqslant l(g)$, or $ht(l(g) - l(g^2)) = 1$. Since $l_k = \pi_k \circ l$ and $\pi_k$ is an order-preserving homomorphism of $\Z^n$ onto $\Z^{n-k}$, from $l(g^2) \geqslant l(g)$ we obtain $l_k(g^2) \geqslant l_k(g)$, and from $ht(l(g) - l(g^2)) = 1$ we obtain $l_k(g) - l_k(g^2) = 0$. Hence, $l_k(g^2) \geqslant l_k(g)$ for any $g \in G$.
\end{proof}

Next, let $k \in (0, n)$ and $m = k + 1$. In this case we have a $\Z$-tree $\Gamma^{k+1}_k$ equipped with an action of $G_{k+1}$ corresponding to the Lyndon length function $l_k : G_{k+1} \to \Z$. Without loss of generality we can assume that $\Gamma^{k+1}_k$ is minimal with respect to the underlying action of $G_{k+1}$. In particular, it implies that $\Gamma^{k+1}_k$ is spanned by the orbit of the basepoint $x^{k+1}_k$, that is
$$\Gamma^{k+1}_k \subseteq \bigcup_{g \in G_{k+1}} [x^{k+1}_k, g x^{k+1}_k].$$
Recall that some vertices of $\Gamma^{k+1}_k$ are labeled by left cosets of $G_k$ in $G_{k+1}$. We are going to modify $\Gamma^{k+1}_k$ so that all vertices of the new $\Z$-tree are labeled by left cosets of $G_k$ in $G_{k+1}$ - this is important for our further investigations.

\begin{lemma}
\label{lem:tree_acted_on}
For any $k \in (0, n)$ there exists a $\Z$-tree $\Gamma_k$, whose vertices are labeled by the cosets of $G_k$ in $G_{k+1}$, upon which $G_{k+1}$ acts without inversions by left multiplication.
\end{lemma}
\begin{proof}
Recall that the $d^{k+1}_k : \Gamma^{k+1}_k \to \Z$ is the distance function on the $\Z$-tree $\Gamma^{k+1}_k$. The base vertex of $\Gamma^{k+1}_k$ is $x^{k+1}_k$. Notice that $x^{k+1}_k \in \Gamma^{k+1}_k$ is labeled by the coset $G_k$.

We define $\Gamma_k$ as follows. The idea is to take $\Gamma^{k+1}_k$ and remove from it all vertices not labeled by cosets of $G_k$ in $G_{k+1}$. Since every non-labeled vertex of $\Gamma^{k+1}_k$ is not a branch point, the structure of $\Gamma_k$ remains essentially the same. More formally, we define the set of vertices $V(\Gamma_k)$ of $\Gamma_k$ to be the set of left cosets of $G_k$ in $G_{k+1}$, and the pair $(g G_k, h G_k)$ is an edge $e \in E(\Gamma_k)$ if and only if the geodesic $[x, y]$ in $\Gamma^{k+1}_k$, where $x$ is labeled by $g G_k$ and $y$ is labeled by $h G_k$, contains no labeled vertices except the endpoints $x$ and $y$. We can consider a natural map $\phi : \Gamma_k \to \Gamma^{k+1}_k$ defined as follows: for every $g G_k \in V(\Gamma_k)$ we set $\phi(g G_k)$ to be the unique vertex of $\Gamma^{k+1}_k$ labeled by $g G_k$, and for every $e = (g G_k, h G_k) \in E(\Gamma_k)$ we set $\phi(e) = [x, y]$, where $x = \phi(g G_k),\ y = \phi(h G_k)$. 

From the definition above we can immediately obtain some basic properties of $\Gamma_k$ and $\phi$. Suppose $v \in V(\Gamma^{k+1}_k)$. If $v$ is labeled by a coset, then, by definition, it has a unique preimage $\phi(v) \in V(\Gamma_k)$. If $v$ is not labeled, then $v$ cannot be a branch point, so from the assumption that $\Gamma^{k+1}_k$ is minimal it follows that there exists a geodesic path $p_v$ in $\Gamma^{k+1}_k$, unique up to switching end-points, that contains $v$ and such that only the initial and terminal vertices of $p_v$ are labeled by cosets. In this case, $p_v$ is the image of some edge in $\Gamma_k$.
Now, if $e_1, e_2 \in \Gamma^{k+1}_k$ share a labeled vertex $v \in V(\Gamma^{k+1}_k)$ and $e_1 \neq e_2,\ e_1 \neq e_2^{-1}$ then there exist edges $f_1, f_2 \in E(\Gamma_k)$ which share $\phi^{-1}(v)$ and such that $f_1 \neq f_2,\ f_1 \neq f_2^{-1}$. In particular, the preimage of a branch point in $\Gamma^{k+1}_k$ is a branch point in $\Gamma_k$. Also, it is easy to see that if $f = \phi(e)$, then $f^{-1} = \phi(e^{-1})$.

First of all, we show that $\Gamma_k$ is a $\Z$-tree. 

Let $a G_k, b G_k \in V(\Gamma_k)$. Since $\Gamma^{k+1}_k$ is a a geodesic metric space, there is a path $p$ in $\Gamma^{k+1}_k$ connecting $\phi(a G_k)$ and $\phi(b G_k)$. This path can be split into a unique sequence of subpaths $p_1, p_2, \ldots, p_m$ so that $t(p_i) = o(p_{i+1})$ for every $i \in [1, m - 1]$, and only the initial and terminal vertices of each $p_i$ are labeled by cosets. Now, if $o(p_i) = a_i G_k$ and $t(p_i) = b_i G_k$, then there exists an edge $q_i = (a_i G_k, b_i G_k) \in E(\Gamma_k)$ such that $\phi(q_i) = p_i$. Hence, the path $q_1 \ldots q_m$ connects $a G_k$ and $b G_k$ in $\Gamma_k$. In other words,  $\Gamma_k$ is a geodesic space.

Suppose there is a circuit $\mathcal{C}$ in $\Gamma_k$, which is a closed path $e_1 e_2 \ldots e_r$. Without loss of generality, we can assume that $r$ is the minimum possible length of a circuit in $\Gamma_k$, in particular, $\mathcal{C}$ has no self-intersections and no cyclic permutation of $\mathcal{C}$ has backtracking. The image of $\mathcal{C}$ under the map $\phi$ is a circuit $\mathcal{C}' = f_1 f_2 \ldots f_r$ in $\Gamma^{k+1}_k$, where each $f_i = \phi(e_i)$ is a geodesic path in $\Gamma^{k+1}_k$ such that only the initial and terminal vertices of each $f_i$ are labeled by cosets. Since $\Gamma^{k+1}_k$ is a $\Z$-tree, $\mathcal{C}'$ tracks a finite subtree $T$ of $\Gamma^{k+1}_k$. Hence, there exists a vertex $x \in T$ which has valence $1$ in $T$. Notice that $x$ is labeled by a coset: indeed, if $x$ is not labeled, then it belongs to some $f_i$ and since $f_i$ is a geodesic path, $x$ cannot have valence $1$ in $T$. Now, since $\mathcal{C}'$ is a closed path, we have $x = t(f_i) = o(f_{i+1})$ for some $i \in [1, r - 1]$ (or, $x = t(f_r) = o(f_1)$). But then, from the fact that $T$ is a tree, it follows that $f_{i+1} = f_i^{-1}$, and it implies that $e_{i+1} = e_i^{-1}$, that is, $\mathcal{C}$ has backtracking, which is a contradiction with our assumption.

So, we can conclude that $\Gamma_k$ is a $\Z$-tree. Let us denote by $d_k$ the $\Z$-metric on $\Gamma_k$.

Since $G_{k+1}$ acts on the set of left cosets of $G_k$ in $G_{k+1}$ by left multiplication, $G_{k+1}$ naturally acts on the set of vertices of $\Gamma_k$, which induces the action of $G_{k+1}$ on $\Gamma_k$. This action corresponds to the action of $G_{k+1}$ on the set of labeled vertices of $\Gamma^{k+1}_k$: if $a G_k, b G_k \in V(\Gamma_k)$ and $x = \phi(a G_k),\ y = \phi(b G_k)$, then for any $g \in G_{k+1}$ we have $d^{k+1}_k(g x, g y) = d^{k+1}_k(x, y)$, which implies that $d_k((g a) G_k, (g b) G_k) = d_k(a G_k, b G_k)$. That is the action of $G_{k+1}$ on $\Gamma_k$ is isometric.

It remains to prove that $G_{k+1}$ acts on $\Gamma_k$ without inversions. Suppose, on the contrary, that there exists $g \in G_{k+1}$ and an edge $e = (h_1 G_k, h_2 G_k) \in E(\Gamma_k)$ such that $g e = e^{-1}$, that is, 
$g (h_1 G_k) = h_2 G_k$ and $g (h_2 G_k) = h_1 G_k$. Notice that $g^2$ fixes $e$.

Since $\Gamma_k$ is a tree, either $h_1 G_k$ belongs to the geodesic $[G_k, h_2 G_k]$, or $h_2 G_k$ belongs to the geodesic $[G_k, h_1 G_k]$. Assume, without loss of generality, that $h_1 G_k$ belongs to $[G_k, h_2 G_k]$. It follows that $g (h_1 G_k)$ belongs to the geodesic $[g G_k, g (h_2 G_k)]$, that is, $h_2 G_k$ belongs to $[g G_k, h_1 G_k]$. Therefore,
$$[G_k, g G_k] = [G_k, h_1 G_k] \cup [h_1 G_k, h_2 G_k] \cup [h_2 G_k, g G_k]$$
and 
$$d_k(G_k, g G_k) = d_k(G_k, h_1 G_k) + d_k(h_1 G_k, h_2 G_k) + d_k(h_2 G_k, g G_k).$$
Since the action of $G_{k+1}$ on $\Gamma_k$ is isometric, we have
$$d_k(G_k, h_1 G_k) = d_k(g G_k, h_2 G_k) = d_k(g^2 G_k, h_1 G_k),$$
so
$$d_k(G_k, g G_k) > d_k(G_k, h_1 G_k) + d_k(g G_k, h_2 G_k)$$
$$= d_k(G_k, h_1 G_k) + d_k(g^2 G_k, h_1 G_k) \geqslant d_k(G_k, g^2 G_k).$$
But then, in terms of $\Gamma^{k+1}_k$, we have 
$$l_k(g) = d^{k+1}_k(x^{k+1}_k, g x^{k+1}_k) > d^{k+1}_k(x^{k+1}_k, g^2 x^{k+1}_k) = l_k(g^2)$$
which contradicts Corollary \ref{cor:0-hyp}. Hence, no inversion is possible.
\end{proof}

\begin{lemma}
\label{HNN extension}
For every $k \in [1, n)$ there exist collections $\{C_i \mid i \in I\}$ and $\{D_i \mid i \in I\}$ of subgroups of $G_k$ (the collections may be infinite) and isomorphisms $\{\phi_i : C_i \rightarrow D_i \mid i \in I\}$ such that 
$$G_{k+1} = \langle G_k, \{h_i \mid i \in I\} \mid h_i^{-1} c h_i = \phi_i(c),\ c \in C_i \rangle.$$ 
Moreover, $C_i = G_k \cap h_i G_k h_i^{-1}$ and $D_i = G_k \cap h_i^{-1} G_k h_i$.
\end{lemma}
\begin{proof} 
By Lemma \ref{lem:tree_acted_on}, there exists a $\Z$-tree $\Gamma_k$, whose vertices are left cosets of $G_k$ in $G_{k+1}$ and $G_{k+1}$ acts on it without inversions. Since $G_{k+1}$ acts transitively, $G_{k+1}$ splits into a graph of groups with a single vertex $v_0$ and multiple loops attached, that is, $G_{k+1}$ is isomorphic to an HNN extension with multiple stable letters.

Consider the stabilizer of a vertex $h G_k$ of $\Gamma_k$. We have
$$g (h G_k) = h G_k \Leftrightarrow h^{-1} g h \in G_k \Leftrightarrow g \in h G_k h^{-1},$$ 
so the stabilizer of $h G_k$ is $h G_k h^{-1}$. Every loop at $v_0$ is the image of an edge $(G_k, h G_k) \in E(\Gamma_k)$ for some $h \in G_{k+1}$. The stabilizer of the vertex $G_k$ is $G_k$ itself, so the stabilizer of the edge $(G_k, h G_k)$ is $G_k \cap h G_k h^{-1}$. This group embeds directly into both $G_k$ and $h G_k h^{-1}$, and the latter is sent to the basepoint via conjugation by $h$, hence, the required splitting of $G_{k+1}$.
\end{proof}

\begin{corollary}
\label{finite HNN}
If $G_{k+1}$ is finitely generated, then $G_{k+1}$ is an HNN extension of $G_k$ with finitely many stable letters.
\end{corollary}
\begin{proof}
By Lemma \ref{HNN extension} we have
$$G_{k+1} = \langle G_k, \{h_i \mid i \in I\} \mid h_i^{-1} c h_i = \phi_i(c),\ c \in C_i \rangle.$$ 
Now, since $G_{k+1}$ is finitely generated, from \cite[Proposition 1.35]{Cohen} it follows that $\{h_i \mid i \in I\}$ is a finite set.
\end{proof}

Let us fix some $k \in [1, n)$. According to Lemma \ref{HNN extension}, we have
$$G_{k+1} = \langle G_k, \{h_i \mid i \in I\} \mid h_i^{-1} c h_i = \phi_i(c),\ c \in C_i \rangle,$$
where $C_i = G_k \cap h_i G_k h_i^{-1}$. Next, we would like to clarify the structure of the associated subgroups $\{C_i \mid i \in I\}$. Fix some $C_i$ and consider the action of $C_i$ on the $\Z^{n-1}$-tree $\Gamma^n_1$ with the distance function $d_1$. Recall that $l_1 : G \to \Z^{n - 1}$ is the corresponding length function based at the point $x^n_1 \in \Gamma^n_1$ and
$$c_1(g, h) = \frac{1}{2}\left(l_1(g) + l_1(h) - l_1(g^{-1} h)\right)$$
for any $g, h \in G$. Since $l_1$ takes values in $\Z^{n-1}$ and not in $\Z^n$, we also introduce the height function $ht_1$ with respect to $l_1$, which is going to be useful when working with distances in $\Gamma^n_1$: for $a \in \Z^{n-1}$ represented by an $(n-1)$-tuple $(a_1, \ldots, a_{n-1}) \in \Z^{n-1}$, we write $ht_1(a) = k$ if $a = (a_1, \ldots,$ $a_k, 0, \ldots, 0)$ and $a_k \neq 0$. For $g \in G$, by $ht_1(g)$ we understand $ht_1(l_1(g))$. There is an obvious connection between the functions $ht$ and $ht_1$: for $g \in G$, we have $ht(g) = k + 1$ if and only if $ht_1(g) = k$.

Recall that some vertices of $\Gamma^n_1$ are labeled by left cosets of $G_1$ in $G$.

\begin{lemma}
\label{co-linear Ci}
Let $c \in C_i$ and $h_i$ be the stable letter associated with $C_i$. Then, either $l_1(c) = c_1(c, c^{-1}) + c_1(c, h_i)$, or $l_1(c) = c_1(c, c^{-1}) + c_1(c^{-1}, h_i)$.
\end{lemma}
\begin{proof}
Notice that $c \in C_i = G_k \cap h_i G_k h_i^{-1}$, so, $c \in G_k$ and $h_i^{-1} c h_i \in G_k$ imply that $ht(c) \leqslant k$ and $ht(h_i^{-1} c h_i) \leqslant k$. Next, $ht(h_i) = k + 1$ since $h_i \notin G_k$ and since $c(g, h) \leqslant \min\{ l(g), l(h) \}$ in general, it follows that $c(c^{-1}, h_i) \leqslant l(c)$, which implies that $ht(c(c^{-1}, h_i)) \leqslant k$. Now, from
$$l(h_i^{-1} c h_i) = 2l(h_i) + l(c) - 2c(h_i, c h_i) - 2c(c^{-1}, h_i)$$
it follows that $ht(c(h_i, c h_i)) = k + 1$. In particular, $ht(c(h_i, c h_i)) > ht(c(c, c h_i))$ since $ht(c(c, c h_i)) \leqslant k$, and, passing to $l_1$ and $c_1$, it implies that
$$c_1(h_i, c h_i) > c_1(c, c h_i).$$
Hence,
$$2c_1(c, h_i) \geqslant 2\min\{c_1(c, c h_i), c_1(h_i, c h_i)\} = 2c_1(c, c h_i) = l_1(c) + l_1(c h_i) - l_1(h_i)$$
$$= 2l_1(c) - 2c_1(c^{-1}, h_i)$$
and we obtain
$$ l_1(c) \leqslant c_1(c, h_i) + c_1(c^{-1}, h_i).$$
It follows that either $c_1(c, h_i) \geqslant \frac{1}{2} l_1(c)$, or $c_1(c^{-1}, h_i) \geqslant \frac{1}{2} l_1(c)$. Without loss of generality we assume that the former holds. Also, notice that $ht(c(h_i, c^{-1} h_i)) = k + 1$ can be obtained using an argument similar to the one which produced $ht(c(h_i, c h_i)) = k + 1$.

\begin{figure}[h]
\centering
\includegraphics[width=2.5in]{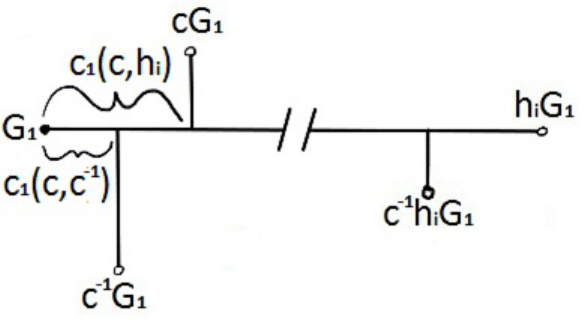}
\caption{The midpoint of $\Delta_2$ and $\Delta_3$}
\label{fig-1}
\end{figure}

Now, consider the geodesic tripods 
$$\Delta_1 = \{c G_1, G_1, h_i G_1\}\ \ {\rm and}\ \ \Delta_2 = \{G_1, c^{-1} G_1, c^{-1} h_i G_1\}$$ 
in $\Gamma^n_1$. Notice that $\Delta_2$ is a translation of $\Delta_1$ by means of $c^{-1}$. The midpoint $Y(c G_1, G_1, h_i G_1)$ of $\Delta_1$ lies on the segment $[G_1, c G_1]$ at the distance of $c_1(c, h_i)$ from the point $G_1$. It follows that the midpoint $Y(G_1, c^{-1} G_1, c^{-1} h_i G_1)$ of $\Delta_2$ lies on the segment $[c^{-1} G_1, G_1]$ at the distance of $c_1(c, h_i)$ from the point $c^{-1} G_1$. Since $l_1(c^2) \geqslant l_1(c)$ by Corollary \ref{cor:0-hyp}, we obtain 
$$c_1(c, c^{-1}) \leqslant \frac{1}{2} l_1(c) \leqslant c_1(c, h_i) < c_1(h_i, c^{-1} h_i),$$ 
so it follows that the midpoint of $\Delta_2$ coincides with the midpoint $Y(G_1, c G_1, c^{-1} G_1)$ of the tripod $\Delta_3 = \{G_1, c G_1, c^{-1} G_1\}$ (see Figure \ref{fig-1}). Hence, the midpoint of $\Delta_2$ lies on the segment $[G_1, c^{-1} G_1]$ at the distance of $c_1(c, c^{-1})$ from the point $G_1$ and eventually we have
$$l_1(c) = l_1(c^{-1}) = d_1(G_1, c^{-1} G_1) =  c_1(c, c^{-1}) + c_1(c, h_i).$$
A similar argument under the assumption $c_1(c^{-1}, h_i) \geqslant \frac{1}{2} l_1(c)$ leads to the equality
$l_1(c) = c_1(c, c^{-1}) + c_1(c^{-1}, h_i)$.
\end{proof}

Recall that the translation length $\ell : G \to \Z^{n - 1}$ related to the action of $G$ on $\Gamma^n_1$ is defined as follows: for any $g \in G$ we have $\ell(g) = \min\{ d_1(x, gx) \mid x \in \Gamma_1^n \}$. Also, recall that $H \leqslant G$ has an {\em abelian action} on $\Gamma^n_1$ if $\ell(g h) \leqslant \ell(g) + \ell(h)$ for any $g, h \in H$ (see \cite[Chapter 3.2]{Chiswell:2001} for details).

\begin{lemma}
\label{abelian action}
The action of each $C_i$ on $\Gamma^n_1$ is abelian. Moreover, each $C_i$ has an abelian action on a $\Z^{k - 1}$-tree.
\end{lemma}
\begin{proof}
Recall that $h_i$ is the stable letter associated with $C_i$. From Lemma \ref{le:conditions_1}, it follows that every $c \in C_i$ acts on $\Gamma_1^n$ either elliptically, or hyperbolically, but not as an inversion.

\begin{figure}[h]
\centering
\includegraphics[width=3.2in]{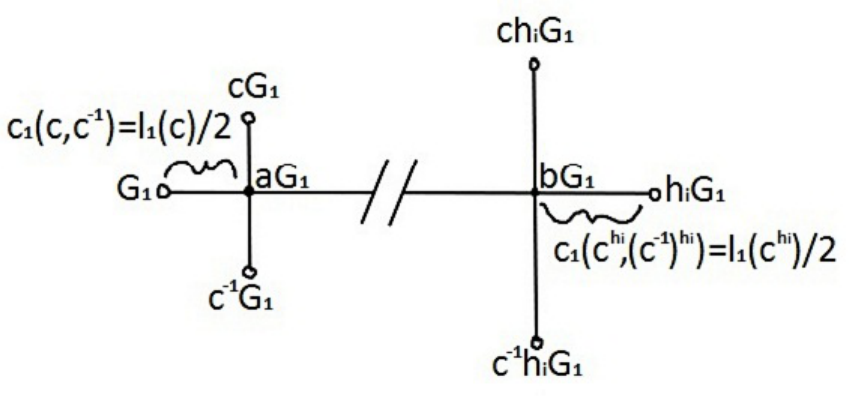}
\caption{The elliptic case}
\label{fig-2}
\end{figure}

Let $c$ act elliptically. Consider the geodesic tripod $\{G_1, c G_1, c^{-1} G_1\}$. Since every branch point of $\Gamma_1^n$ is labeled by a left coset of $G_1$ in $G$, let $a G_1$ be the midpoint of $\{G_1, c G_1, c^{-1} G_1\}$. Consider the tripod $\{a G_1, (c h_i) G_1, h_i G_1\}$ and let $b G_1$ be its midpoint (both tripods and their midpoints are shown in Figure \ref{fig-2}). From \cite[Lemma 3.1.1]{Chiswell:2001} it follows that $a G_1$ is fixed by $c$. Now, the geodesic segment $[a G_1, h_i G_1]$ contains the point $b G_1$ and its image under the action of $c$ is the segment $[a G_1, (c h_i) G_1]$ of the same length. Since $\Gamma_1^n$ is a tree, the point $b G_1$ is fixed by $c$. So, $[a G_1, b G_1] \subseteq Axis(c)$. Moreover, $ht_1(d_1(a G_1, b G_1)) = k$ since $ht_1(c(h_i, c h_i)) = k$, which follows from $ht(c(h_i, c h_i)) = k + 1$ (we refer the reader to the beginning of the proof of Lemma \ref{co-linear Ci} for details) and $ht_1(d_1(G_1, a G_1)) = k - 1$.

\begin{figure}[h]
\centering
\includegraphics[width=3.2in]{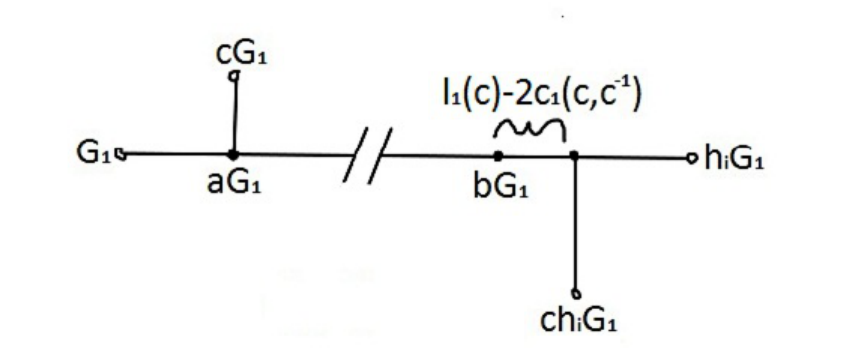}
\caption{The hyperbolic case}
\label{fig-3}
\end{figure}

Let $c$ act hyperbolically. Hence, we have $\ell(c) = l_1(c) - 2c_1(c, c^{-1})$ (see \cite[Theorem 3.1.4]{Chiswell:2001}). By Lemma \ref{co-linear Ci}, either $l_1(c) = c_1(c, c^{-1}) + c_1(c, h_i)$, or $l_1(c) = c_1(c, c^{-1}) + c_1(c^{-1}, h_i)$. Assume, without loss of generality, that the former equality holds and consider the geodesic tripod $\{G_1, c G_1, h_i G_1\}$ with the midpoint $a G_1$ (otherwise, if the latter equality holds we consider the tripod $\{G_1, c^{-1} G_1, h_i G_1\}$ and use an argument similar to the one below). Notice that the midpoint $a G_1$ belongs to the geodesic segment $[G_1, h_i G_1]$. But it also belongs to the segment $[c G_1, (c h_i) G_1]$. Indeed, otherwise we have $c_1(c, c h_i) > c_1(c, h_i) = c_1(h_i, c h_i)$ and we obtain a contradiction with the fact that $ht(c(h_i, c h_i)) = k + 1$ and $ht(c(c, c h_i)) = k$. Next, $d_1(G_1, a G_1) = c_1(c, h_i)$, and from
$$l_1(c) = d_1(G_1, c G_1) = d_1(G_1, a G_1) + d_1(a G_1, c G_1)$$
and 
$$l_1(c) = c_1(c, c^{-1}) + c_1(c, h_i)$$
we obtain $d_1(a G_1, c G_1) = c_1(c, c^{-1})$. Now, both $a G_1$ and $(c a) G_1$ belong to the geodesic segment $[c G_1, (c h_i) G_1]$. Notice that $(c a) G_1$ cannot be closer to $c G_1$ than $a G_1$. Indeed, we have 
$$d_1(c G_1, (c a) G_1) = d_1(G_1, a G_1)  = d_1(G_1, a G_1) - d_1(a G_1, c G_1)$$
$$= l_1(c) - c_1(c, c^{-1}),$$
$$d_1(a G_1, c G_1) = c_1(c, c^{-1})$$
and if we assume that $d_1(c G_1, (c a) G_1) < d_1(a G_1, c G_1)$, then we obtain $l_1(c) - c_1(c, c^{-1}) < 
 c_1(c, c^{-1})$, which is equivalent to $l_1(c) > l_1(c^2)$, which is a contradiction. Hence, $d_1(c G_1, (c a) G_1) > d_1(a G_1, c G_1)$ (the equality $d_1(c G_1, (c a) G_1) = d_1(a G_1, c G_1)$ is impossible since it implies that $a G_1 = (c a) G_1$, that is, $c$ fixes the point $a G_1$, which is a contradiction) and it follows that $a G_1$ belongs to the segment $[c G_1, (c a) G_1]$ and
$$d_1(a G_1, (c a) G_1) = d_1(c G_1, (c a) G_1) - d_1(a G_1, c G_1)$$
$$= l_1(c) - c_1(c, c^{-1}) - c_1(c, c^{-1}) = l_1(c) - 2c_1(c, c^{-1}) = \ell(c).$$
It follows that the point $a G_1$ belongs to $Axis(c)$. Finally, consider the geodesic tripod $\{G_1, h_i G_1, (c h_i) G_1\}$ and let its midpoint be $f G_1$. Now, consider a point $v$ on the geodesic segment $[G_1, f G_1]$ at the distance of $l_1(c) - 2c_1(c, c^{-1})$ from $f G_1$. Notice that the segment $[a G_1, v]$ is translated by $c$ to the segment $[(c a) G_1, f G_1]$, that is, $c v = f G_1$ and it follows that $v$ is labeled by a coset $b G_1$ (see  Figure \ref{fig-3}). Since we have $c (b G_1) = f G_1$, we can take $f = c b$ and by the choice of $v$ we have 
$$d_1(b G_1, (c b) G_1) =  l_1(c) - 2c_1(c, c^{-1}) = \ell(c).$$

This means that the point $b G_1$ also belongs to $Axis(c)$, therefore, $[a G_1, b G_1] \subseteq Axis(c)$. Also notice that $ht_1(d_1(a G_1, b G_1)) = k$ since $ht_1(c(h_i, c h_i)) = k$ (which follows from $ht(c(h_i, c h_i)) = k + 1$) and $ht_1(d_1(G_1, a G_1)) = ht_1(d_1(b G_1, (c b) G_1)) = k - 1$.

\smallskip

Notice that both in the elliptic, and hyperbolic case, $Axis(c)$ contains a sub-segment $I_c$ of $[G_1, h_i G_1]$, whose length has $l_1$-height equal to $k$ and the distance of one of its endpoints from $G_1$ has $l_1$-height equal to $k - 1$. It follows that for any $c_1, c_2 \in C_i$ we have
$$\text{diam}(Axis(c_1) \cap Axis(c_2)) \geqslant \text{diam}(I_{c_1} \cap I_{c_2})$$ 
and $ht_1(\text{diam}(I_{c_1} \cap I_{c_2})) = k$. In particular, 
$$Axis(c_1) \cap Axis(c_2) \neq \varnothing$$
for any $c_1, c_2 \in C_i$. Since, by Lemma \ref{le:conditions_1}, elements of $G$ cannot act on $\Gamma^n_1$ as inversions, from \cite[Corollary 3.2.4]{Chiswell:2001}, it follows that 
$$\ell(c_1 c_2) \leqslant \ell(c_1) + \ell(c_2),$$
so the action of $C_i$ on $\Gamma^n_1$ is abelian.

\smallskip

Finally, since $C_i \subseteq G_k$ and the restriction of $l_1$ to $G_k$ takes values in $\Z^{k-1}$, it follows that $C_i$ has an abelian action on a $\Z^{k - 1}$-subtree of $\Gamma^n_1$ containing the vertex labeled by $G_1$.
\end{proof}

\begin{corollary}
\label{define elliptic}
Let $E_i$ be the set of elements of $C_i$ acting elliptically on $\Gamma^n_1$. Then $E_i \unlhd C_i$ and $C_i / E_i$ is a free abelian group of rank at most $k - 1$.
\end{corollary}
\begin{proof}
By Lemma \ref{abelian action}, the action of $C_i$ on $\Gamma^n_1$ is abelian. Hence, by \cite[Proposition 3.2.7]{Chiswell:2001}, there exists a homomorphism $\rho_i : C_i \to \Z^{n - 1}$ such that $\ell(c) = |\rho_i(c)|$ for every $c \in C_i$. In fact, since $C_i \subseteq G_k$ and $\ell(c) = l_1(c) - 2c_1(c, c^{-1})$, it follows that $\ell(c) \in \Z^{k-1}$ since the restriction of $l_1$ to $G_k$ takes values in $\Z^{k-1}$. The kernel of $\rho_i$ is the set of elements $c \in C_i$ whose translation length $\ell(c) = 0$, so $ker(\phi_i) = E_i$ and $C_i / E_i$ embeds into $\Z^{k-1}$, that is, it is a free abelian group of rank at most $k - 1$.
\end{proof}

Now we can formulate and prove the auxiliary result we mentioned in the beginning of this subsection.

\begin{prop}
\label{pr:main_technical}
Let $G$ be a torsion-free group with a $\delta$-hyperbolic $\delta$-regular length function $l : G \rightarrow \Z^n$, where $\Z^n$ has the right lexicographic order. Let $G_k = \{g \in G \mid ht(g) \leqslant k\}$ so that $G$ is the union of the finite chain of subgroups
$$1 = G_0 \leqslant G_1 \leqslant \cdots \leqslant G_n = G.$$ 
If the following conditions are satisfied
\begin{enumerate}
\item[(a)] $ht(\delta) = 1$,
\item[(b)] $l(g) > 0$ for every non-trivial $g \in G$,
\item[(c)] $G_1$ is a finitely generated isolated subgroup of $G$,
\end{enumerate}
then for every $k \in [1, n)$, the group $G_{k+1}$ is isomorphic to an HNN extension of $G_k$, where every associated subgroup has an abelian action on a $\Z^{k-1}$-tree. Moreover, if $G_{k+1}$ is finitely generated then the HNN extension for $G_{k+1}$ has only a finite number of stable letters.
\end{prop}
\begin{proof}
The proof is just a combination of the results proved in this subsection. 

By Lemma \ref{HNN extension}, for every $k \in [1, n)$ we have a splitting of $G_{k+1}$ into an HNN extension
$$G_{k+1} = \langle G_k, \{h_i \mid i \in I\} \mid h_i^{-1} c h_i = \phi_i(c),\ c \in C_i \rangle,$$
where $C_i = G_k \cap h_i G_k h_i^{-1}$. The set of indices $I$ is finite if $G_{k+1}$ is finitely generated (by Corollary \ref{finite HNN}). Finally, from Lemma \ref{abelian action} it follows that each $C_i$ in the HNN extension above has an abelian action on a $\Z^{k - 1}$-tree.
\end{proof}

\subsection{Proof of the main theorem}
\label{sec:special}

The goal of this subsection is to prove the following theorem stated in the introduction.

\begin{theorem}
\label{th:main_result}
Let $G$ be a finitely generated torsion-free group with a proper $\delta$-hyperbolic $\delta$-regular length function $l : G \rightarrow \Z^n$, where $\Z^n$ has the right lexicographic order. Let $G_k = \{g \in G \mid ht(g) \leqslant k\}$ so that $G$ is the union of the finite chain of subgroups
$$1 = G_0 \leqslant G_1 \leqslant \cdots \leqslant G_n = G.$$ 
If the following conditions are satisfied
\begin{enumerate}
\item[(a)] $ht(\delta) = 1$,
\item[(b)] $l(g) > 0$ for every non-trivial $g \in G$,
\item[(c)] $G_1$ is a finitely generated isolated subgroup of $G$,
\end{enumerate}
then $G_1$ is a word-hyperbolic group and for every $k \in [1, n)$, the group $G_{k+1}$ is isomorphic to an HNN extension of $G_k$ with a finite number of stable letters, where each associated subgroup is isomorphic to either $\Z^m$, or an extension of $\Z^m$ by $\Z$, where $m \leqslant k - 1$.
\end{theorem}

From now on we assume that all conditions from the statement of Theorem \ref{th:main_result} hold. Notice that the assumptions on $G$ and $l$ of Theorem \ref{th:main_result} imply the assumptions of Proposition \ref{pr:main_technical}.

Recall Example \ref{ex:2} above. Note that the condition (b) of Theorem \ref{th:main_result} is satisfied by the choice of $x_0$ in the example. It follows that every finitely generated group $G$ acting properly and co-compactly on a geodesic $\Z^n$-hyperbolic space $X$ so that the stabilizer of some $\Z$-subspace $X_0$ is finitely generated, by Theorem \ref{th:main_result} can be represented as a finite series of HNN extensions described above. In light of this observation, save for one technical assumption, we can think of groups satisfying the conditions of Theorem \ref{th:main_result} as a higher-dimensional analog of classical hyperbolic groups.

Recall that for every $g \in G_1$ we have $l(g) = (a, 0, \ldots, 0)$ for some positive $a \in \Z$, hence, we can define $d : G_1 \times G_1\rightarrow \Z$ by setting $d(g, h) = k$ if $l(g^{-1} h) = (k, 0, \ldots, 0)$. It is easy to see that $d$ is a $\Z$-metric on $G_1$.

Since we assume that $G_1$ is finitely generated, let $S \subset G_1$ be a generating set for $G_1$ and $\Gamma(G_1, S)$ the Cayley graph of $G$ with respect to $S$.

\begin{lemma}
\label{G1 hyperbolic}
$G_1$ is a torsion-free word-hyperbolic group and $\Gamma(G_1, S)$ with the word metric $\mid \cdot \mid_S$ is quasi-isometric to $(G_1, d)$.
\end{lemma}
\begin{proof}
First of all, $(G_1, d)$ is a $\delta$-hyperbolic ($\Z$-)metric space which follows from $\delta$-hyperbolicity of $l$ and the assumption (a) on $\delta$. Moreover, $(G_1, d)$ is regular since $l$ is regular.

Using the construction of \cite[Section 6.1]{GKMS:2012}, from $(G_1, d)$ we obtain a metric space $\Gamma$, which quasi-isometrically embeds into $(G_1, d)$. $G_1$ acts on $(G_1, d)$ by left multiplication, this action induces an action of $G_1$ on $\Gamma$. If $x$ is an arbitrary point in $\Gamma$, since $l$ is proper, the action of $G_1$ on $\Gamma$ is proper (relative to $x$) too. By \cite[Lemma 6.3]{GKMS:2012}, $\Gamma$ is quasi-isometric to the Cayley graph $\Gamma(G_1, S \cup G_x)$ with respect to the generating $S \cup G_x$, where $G_x = Stab_{G_1}(x)$. Notice that $G_x$ is trivial since $l(g) > 0$ for every non-trivial $g \in G$. It follows that $\Gamma$ is quasi-isometric to $\Gamma(G_1, S)$. Hence, $(G_1, d)$ is quasi-isometric to $\Gamma(G_1, S)$. Finally, since $(G_1, d)$ is hyperbolic and $S$ is finite, it follows that $G_1$ is hyperbolic.
\end{proof}

Recall that for every $k \in (0, n)$, the function $l_k : G \to \Z^{n - k}$ is a Lyndon length function on $G$ and there exists a $\Z^{n - k}$-tree $(\Gamma^n_k, d^n_k)$, an action of $G$ on $\Gamma^n_k$, and a point $x^n_k \in \Gamma^n_k$ such that $l_k(g) = d^n_k(x^n_k, g x^n_k)$ for every $g \in G$. Also recall that some points of $\Gamma^n_k$ are labeled by left cosets of $G_k$ in $G$.

\begin{lemma}
\label{stabilizers cyclic}
Suppose $h \in G \ssm G_k$ and let $h G_k$ be a point in $\Gamma^n_k$. Then $Stab_G(h G_k) \cap G_1$ is either trivial or cyclic.
\end{lemma}
\begin{proof}
First of all, notice that $G_1 \leqslant G_k$ and $G_k$ stabilizes the point $x^n_k \in \Gamma^n_k$ labeled by $G_k$, that is, $G_1 \leqslant Stab_G(G_k)$. 

Suppose $Stab_G(h G_k) \cap G_1$ is non-trivial and take a non-trivial element $g$ from the intersection. Observe that $Stab_G(h G_k) = h G_k h^{-1}$, so we have $h^{-1} g h \in G_k$ and from
$$c(h, g h) = \frac{1}{2} (l(h) + l(g h) - l(h^{-1} g h)$$
we obtain $ht(c(h, g h)) = ht(h) > k$. Since $c(g, g h) \leqslant l(g)$,we obtain $ht(c(g, g h)) \leqslant 1$ and
$$c(g, h) \geqslant \min\{c(g, gh), c(h, gh)\} - \delta = c(g, gh) -\delta = l(g) - c(g^{-1}, h) - \delta.$$
Thus,
$$c(g, h) + c(g^{-1}, h) \geqslant l(g) - \delta.$$
and either 
$$c(g, h) \geqslant \frac{1}{2} \left(l(g) - \delta\right),$$
or 
$$c(g^{-1}, h) \geqslant \frac{1}{2} \left(l(g) - \delta\right).$$
Notice that for every $g \in Stab_G(h G_k) \cap G_1$ we have $l(g) = (a, 0, \ldots, 0)$, where $a$ is a natural number, since $g \in G_1$. It follows that $c(g, h) \in \R$ since $c(g, g h) \leqslant l(g)$.

Let $\{x_m\}$ and $\{y_m\}$ be two sequences of elements of $Stab_G(h G_k) \cap G_1$ such that $\{c(x_m, h)\} \to \infty$ and $\{c(y_m, h)\} \to \infty$ as $m \to \infty$. At least one such sequence exists. Indeed, we assumed that $Stab_G(h G_k) \cap G_1$ is non-trivial, so $Stab_G(h G_k) \cap G_1$ is infinite because $G_1$ is torsion-free. Since $(G_1, d)$ is a proper metric space, the lengths of elements of $Stab_G(h G_k) \cap G_1$ are not bounded by any constant, so it is possible to find a sequence of elements $\{z_q\} \subseteq Stab_G(h G_k) \cap G_1$ such that $l(z_q) = (a_q, 0, \ldots, 0)$ and $a_q \to \infty$ as $q \to \infty$ (in this case we simply write $l(z_q) \to \infty$). Finally, using the inequalities
$$c(g, h) \geqslant \frac{1}{2} \left(l(g) - \delta\right),\ \ \ \ c(g^{-1}, h) \geqslant \frac{1}{2} \left(l(g) - \delta\right),$$
at least one of which holds for $g$, and taking subsequences of $\{z_q\}$, we obtain the required.

Notice that since $c(x_m, y_r) \geqslant \min\{c(x_m, h), c(y_r, h)\} - \delta$, and 
$$\lim_{m \to \infty} c(x_m, h) = \infty = \lim_{r \to \infty} c(y_r, h),$$ 
we automatically have that 
$$\lim_{m, r \to \infty} c(x_m, y_r) = \infty.$$ 
Similarly, we have $c(x_m, x_r) \geqslant \min\{c(x_m, h), c(x_r, h)\} - \delta$, so
$$\lim_{m, r \to \infty} c(x_m, x_r) = \infty.$$ 
In other words, $\{x_m\}$ and $\{y_m\}$ converge to the same point of the boundary $\partial(G_1, d)$ of $(G_1, d)$ (which is a $\delta$-hyperbolic proper metric space). Now, if $S$ is a generating set of $G_1$, then, by Lemma \ref{G1 hyperbolic}, $\Gamma(G_1, S)$ is quasi-isometric to $(G_1, d)$, hence, $\{x_m\}$ and $\{y_m\}$ converge to the same point of the boundary $\partial\Gamma(G_1, S)$ of $\Gamma(G_1, S)$. Denote this point by $\alpha$.


Recall that $g \in Stab_G(h G_k) \cap G_1$ is an arbitrary non-trivial element of the intersection and let $\{x_r\}$ be a sequence of elements of $Stab_G(h G_k) \cap G_1$ such that $\lim_{r \to \infty} c(x_r, h) = \infty$. Hence, 
$$c(g x_r, h) \geqslant \min\{c(g x_r, g h), c(g h, h)\} - \delta.$$
Since $g \in Stab(h G_k) \cap G_1$, we have $ht(h^{-1} g h) \leqslant k$, which implies
$$ht(c(g h, h)) = ht(l(g h) + l(h) - l(h^{-1} g h)) \geqslant k + 1 > ht(g x_r).$$
Since $c(g x_r, g h) \leqslant l(g x_r)$, we have $ht(c(g x_r, g h)) \leqslant k$ and it follows that $c(g x_r, g h) < c(g h, h)$. Hence, $c(g x_r, h) \geqslant c(g x_r, g h) - \delta$. Next,
$$c(g x_r, g h) = \frac{1}{2}(l(g x_r) + l(g h) - l(x_r^{-1} h))$$
$$\geqslant \frac{1}{2}(l(x_r) - l(g) + l(h) - l(g) - l(x_r^{-1} h)$$
$$= \frac{1}{2}(l(x_r) + l(h) - l(x_r^{-1} h) - 2 l(g)) = c(x_r, h) - l(g).$$
It follows that 
$$\lim_{r \to \infty} c(g x_r, h) \geqslant \lim_{r \to \infty} (c(x_r, h) - l(g)) = \infty.$$
Hence, $g \alpha = \alpha$ for any $g \in Stab_G(h G_k) \cap G_1$, so 
$$g \in Stab_G(h G_k) \cap G_1 \subseteq Stab_{G_1}(\alpha).$$ 
From \cite[Theorem 8.30]{Ghys_delaHarpe:1991}) it follows that $Stab_{G_1}(\alpha)$ is virtually cyclic, and since $G_1$ is torsion-free, so is $Stab_{G_1}(\alpha)$. Finally, by \cite[Lemma 3.2]{Macpherson:1996}), $Stab_{G_1}(\alpha)$ is cyclic, so is $Stab_G(h G_k) \cap G_1$.
\end{proof}

Recall that $\Gamma^n_1$ is a $\Z^{n-1}$-tree, which $G$ acts upon with the underlying length function $l_1 : G \to \Z^{n - 1}$, and all branch points of $\Gamma_1^n$ (possibly some other points too) are labeled by left cosets of $G_1$ in $G$. Also recall that $d_1 : \Gamma_1^n \times \Gamma_1^n \to \Z^{n-1}$ is the distance function on $\Gamma_1^n$.

By Lemma \ref{HNN extension}, $G_{k+1}$ is an HNN extension of $G_k$ with finitely many stable letters $h_i$ and associated subgroups $C_i = G_k \cap h_i G_k h_i^{-1}$. For every such $C_i$, in Corollary \ref{define elliptic} we defined $E_i$ to be the set of elements of $C_i$ acting elliptically on $\Gamma^n_1$, moreover, we proved that $E_i$ is a normal subgroup of $C_i$ and $C_i / E_i$ is a free abelian group of rank at most $k - 1$. Now, using our new assumptions on $G$ and $l$ we can further clarify the structure of $E_i$ and $C_i$.

\begin{lemma}
\label{elliptic cyclic}
$E_i$ is either trivial or an infinite cyclic group.
\end{lemma}
\begin{proof}
Notice that if $g \in E_i$, then $l_1(g) - 2c_1(g, g^{-1}) = \ell(g) = 0$ and $l_1(g) = 2c_1(g, g^{-1})$. Hence, by Lemma \ref{co-linear Ci}, we have either
$$c_1(g, h_i) = c_1(g, g^{-1}) = \frac{1}{2} l_1(g)$$
or
$$c_1(g^{-1}, h_i) = c_1(g, g^{-1}) = \frac{1}{2}  l_1(g).$$
Observe that if $c_1(g, h_i) = c_1(g, g^{-1})$, then 
$$c_1(g^{-1}, h_i) \geqslant \min\{c_1(g, h_i), c_1(g, g^{-1})\} = c_1(g, g^{-1}) = \frac{1}{2} l_1(g),$$
and if $c_1(g^{-1}, h_i) = c_1(g, g^{-1})$, then
$$c_1(g, h_i) \geqslant \min\{c_1(g^{-1}, h_i), c_1(g, g^{-1})\} = c_1(g, g^{-1}) = \frac{1}{2} l_1(g).$$
In other words, for any $g \in E_i$ we have 
$$c_1(g^{-1}, h_i) \geqslant \frac{1}{2} l_1(g),\ \ \ \ c_1(g, h_i) \geqslant \frac{1}{2} l_1(g).$$
Hence, for any $g_1, g_2 \in E_i$ we obtain
$$c_1(g_1^{-1}, g_2) \geqslant \min\{c_1(g_1^{-1}, h_i), c_1(g_2, h_i)\} \geqslant \frac{1}{2} \min\{l_1(g_1), l_1(g_2)\},$$
that is,
$$l_1(g_1) + l_1(g_2) - l_1(g_1 g_2) \geqslant \min\{l_1(g_1), l_1(g_2)\},$$
which implies that
$$l_1(g_1 g_2) \leqslant \max\{l_1(g_1), l_1(g_2)\}.$$
It follows that for any $g \in E_i$, the set $E_i(g) = \{h \in E_i \mid l_1(h) \leqslant l_1(g)\}$ is a subgroup of $E_i$.

Let $g \in E_i$ and consider the geodesic tripod $\{G_1, g G_1, g^{-1} G_1\}$ in $\Gamma_1^n$. Notice that the midpoint $v_g =  Y(G_1, g G_1, g^{-1} G_1)$ of the tripod is fixed under the action of $g$. Indeed, since $l_1(g) = 2c_1(g, g^{-1})$, we have
$$d_1(G_1, v_g) = c_1(g, g^{-1}) = \frac{1}{2} l_1(g) = \frac{1}{2} d(G_1, gG_1) = \frac{1}{2} d(G_1, g^{-1} G_1),$$
that is, $d_1(G_1, v_g) = d_1(g G_1, v_g) = d_1(g^{-1} G, v_g)$. Now, the segment $[G_1, g^{-1} G_1]$ is translated to $[G_1, g G_1]$ under the action of $g$ and since $v_g$ is the midpoint of both, we have $g v_g = v_g$.

Assume, without loss of generality that
$$c_1(g, h_i) = c_1(g, g^{-1}) = \frac{1}{2} l_1(g)$$
and consider the geodesic tripod $\{G_1, g G_1, h_i G_1\}$ in $\Gamma_1^n$ (otherwise, if $c_1(g^{-1}, h_i) = c_1(g, g^{-1})$, then we can consider the tripod $\{G_1, g^{-1} G_1, h_i G_1\}$ and apply a similar argument). Let $c_g \in G$ be such that $c_g G_1 = Y(G_1, g G_1, h_i G_1)$ is the midpoint of this tripod. Recall that every branch point of $\Gamma_1^n$ is labeled by a coset of $G_1$ in $G$, so such $c_g$ always exists. From $c_1(g, h_i) = c_1(g, g^{-1})$ it follows that $c_g G_1$ coincides with the point $v_g$, that is, $c_g G_1$ is fixed by $g$. 

Next, consider the geodesic tripod $\{c_g G_1, (g h_i) G_1, h_i G_1\}$ in $\Gamma_1^n$ and let $d_g \in G$ be such that $d_g G_1$ is the midpoint of this tripod. Since $c_g G_1$ is fixed by $g$, it follows that $d_g G_1$ is also fixed by $g$ (indeed, the segment $[c_g G_1, h_i G_1]$ is translated to $[c_g G_1, (g h_i) G_1]$ with $c_g G_1$ being fixed, so $d_g G_1$ is fixed too). Furthermore, since $c_g G_1 \in [G_1, h_i G_1]$, we have that $d_g G_1$ is also the midpoint of the geodesic tripod $\{G_1, (g h_i) G_1, h_i G_1\}$. Hence, it belongs to the segment $[G_1, h_i G_1]$ and $d_1(G_1, d_g G_1) = c_1(h_i, g h_i)$.

Then suppose that $h \in E_i(g)$, in other words $h \in E_i$ and $l_1(h) \leqslant l_1(g)$. We can apply the above argument to $h$ instead of $g$ and obtain the points $c_h G_1, d_h G_1$ fixed by $h$ on the segment $[G_1, h_i G_1]$. Notice that since $l_1(h) \leqslant l_1(g)$, it follows that $c_1(h, h^{-1}) \leqslant c_1(g, g^{-1})$ and we have
$$d_1(G_1, c_h G_1) \leqslant d_1(G_1, c_g G_1).$$
It follows that
$$[c_g G_1, d_h G_1] \subseteq [c_h G_1, d_h G_1],$$ 
which means that $h$ fixes $c_g G_1$. Hence, $h \in c_g G_1 c_g^{-1}$ for every $h \in E_i(g)$ and we obtain 
$c_g^{-1} E_i(g) c_g \leqslant G_1$. Furthermore, since $E_i(g) \leqslant E_i \leqslant C_i = G_k \cap h_i G_k h_i^{-1}$, we have $E_i(g) \leqslant h_i G_k h_i^{-1}$, hence, $c_g^{-1} E_i(g) c_g \leqslant c_g^{-1} h_i G_k h_i^{-1} c_g$. It follows that 
$$c_g^{-1} E_i(g) c_g \leqslant G_1 \cap c_g^{-1} h_i G_k h_i^{-1} c_g = G_1 \cap Stab_G((c_g^{-1} h_i) G_k),$$
where $(c_g^{-1} h_i) G_k$ is viewed as a point in $\Gamma^n_k$. Since $c_g \in G_k$ and $h_i \notin G_k$, we have $c_g^{-1} h_i \notin G_k$ and by Lemma \ref{stabilizers cyclic}, $c_g^{-1} E_i(g) c_g$ is either trivial, or infinite cyclic. Hence, $E_i(g)$ is either trivial, or infinite cyclic.

Suppose that there exist non-trivial $g, h \in E_i$ such that $l_1(g) < l_1(h)$. Hence, both $E_i(g)$ and $E_i(h)$ are non-trivial (indeed, $g \in E_i(g)$ and $h \in E_i(h)$) cyclic and $E_i(g) < E_i(h)$. It follows that the index $|E_i(h) : E_i(g)|$ is finite. However, that implies that there exist $m, n \in \Z$ such that $m \neq n$ and $h^m E_i(g) = h^n E_i(g)$. In other words, $h^{m - n} \in E_i(g)$, so, $l_1(h^{m - n}) < l_1(h)$, which is impossible.

It follows that for all non-trivial elements $g, h \in E_i$ we have $l_1(g) = l_1(h)$, implying that either $E_i$ is trivial, or $E_i = E_i(g)$ for some $g$, meaning that $E_i$ is infinite cyclic.
\end{proof}

Now we are ready to prove the main result of the paper.

\begin{proof}[Proof of Theorem \ref{th:main_result}]
Since the assumptions of Theorem \ref{th:main_result} imply the assumptions of Proposition \ref{pr:main_technical}, it follows from Proposition \ref{pr:main_technical} that for every $k \in [1, n)$, the group $G_{k+1}$ is isomorphic to an HNN extension of $G_k$, where every associated subgroup has an abelian action on a $\Z^{k-1}$-tree. That is, we have
$$G_{k+1} = \langle G_k, \{h_i \mid i \in I\} \mid h_i^{-1} c h_i = \phi_i(c),\ c \in C_i \rangle,$$
where $C_i = G_k \cap h_i G_k h_i^{-1}$. From Corollary \ref{define elliptic} it follows that each $C_i$ contains a normal subgroup $E_i$ such that $C_i / E_i$ is a free abelian group of rank at most $k - 1$. By Lemma \ref{elliptic cyclic}, $E_i$ is either trivial, or an infinite cyclic. Hence, $C_i$ is isomorphic to either $\Z^m$ (in the case when $E_i$ is trivial), or an extension of $\Z^m$ by $\Z$ (in the case when $E_i$ is an infinite cyclic), where $m \leqslant k - 1$. In particular, each $C_i$ is finitely generated. Thus, if $G_{k+1}$ is finitely generated, then from  \cite[Proposition 1.35]{Cohen}, it follows that the set of indices $I$ is finite and $G_k$ is also finitely generated. Since $G = G_n$ is finitely generated by assumption, using induction, one easily obtains that $G_k$ is finitely generated for every $k \in [1, n)$. Finally, by Lemma \ref{G1 hyperbolic}, $G_1$ is a torsion-free word-hyperbolic group. 
\end{proof}

\begin{corollary}
Under assumptions of Theorem \ref{th:main_result}, if $n = 2$, then $G$ is relatively hyperbolic with abelian parabolics.
\end{corollary}
\begin{proof}
Since $n = 2$, we have
$$1 \leqslant G_1 \leqslant G_2 = G$$
and by Theorem \ref{th:main_result}, we obtain
$$G_2 = \langle G_1, h_1, \ldots, h_m \mid h_i^{-1} c h_i = \phi_i(c),\ c \in C_i \rangle,$$
where $C_i = G_1 \cap h_i G_1 h_i^{-1}$ and $G_1$ is a torsion-free word-hyperbolic group. It follows that each $C_i$ is an infinite cyclic. Moreover, each $C_i$ is a maximal cyclic subgroup of $G_1$. Indeed, if there exists $a \in G_1$ such that $C_i \lneq \langle a \rangle$, that is, $a \notin C_i$, then $h_i^{-1} a h_i \notin G_1$ (otherwise $a \in G_1 \cap h_i G_1 h_i^{-1} = C_i$) and $l_1(h_i^{-1} c h_i) \neq 0$. But, at the same time, $a^r \in C_i$ for some $r \in \Z$, so it follows that $h_i^{-1} a^r h_i \in G_1$ and $l_1(h_i^{-1} a^r h_i) = 0$, which is a contradiction. Finally, since $G$ is an HNN extension of a hyperbolic group with finitely many stable letters and maximal cyclic associated subgroups, the required result follows from \cite[Theorem 0.1]{Dahmani:2003}.
\end{proof}

\end{document}